\documentclass{article}
\usepackage{authblk}
\usepackage{a4wide}
\usepackage{amsthm}
\usepackage{graphicx}
\usepackage{amssymb}
\usepackage{amsmath}
\usepackage{ascmac}
\usepackage{setspace}
\usepackage{float}
\usepackage{algorithm}
\usepackage{algorithmic}
\usepackage{setspace}
\usepackage{tikz-cd}
\newtheorem{Theorem}[equation]{Theorem}
\newtheorem{Corollary}[equation]{Corollary}
\newtheorem{Lemma}[equation]{Lemma}

\theoremstyle{definition}
\newtheorem{Definition}[equation]{Definition}

\theoremstyle{remark}
\newtheorem{Remark}[equation]{Remark}
\numberwithin{equation}{section}

\DeclareMathOperator{\ev}{ev}
\DeclareMathOperator{\id}{id}

\DeclareMathOperator{\ad}{ad}

\newcommand{\ve}{\varepsilon}

\allowdisplaybreaks
\begin{document}
\title{Two Homomorphisms from the affine Yangian associated with $\widehat{\mathfrak{sl}}(n)$ to the affine Yangian associated with $\widehat{\mathfrak{sl}}(n+1)$}
\author{Mamoru Ueda\thanks{mueda@ualberta.ca}}
\affil{Department of Mathematical and Statistical Sciences, University of Alberta, 11324 89 Ave NW, Edmonton, AB T6G 2J5, Canada}
\date{}
\maketitle
\begin{abstract}
We construct a homomorphism from the affine Yangian $Y_{\hbar,\ve+\hbar}(\widehat{\mathfrak{sl}}(n))$ to the affine Yangian $Y_{\hbar,\ve}(\widehat{\mathfrak{sl}}(n+1))$ which is different from the one in \cite{U8}. By using this homomorphism, we give a homomorphism from $Y_{\hbar,\ve}(\widehat{\mathfrak{sl}}(n))\otimes Y_{\hbar,\ve+n\hbar}(\widehat{\mathfrak{sl}}(m))$ to $Y_{\hbar,\ve}(\widehat{\mathfrak{sl}}(m+n))$. As an application, we construct a homomorphism from the affine Yangian $Y_{\hbar,\ve+n\hbar}(\widehat{\mathfrak{sl}}(m))$ to the centralizer algebra of the pair of affine Lie algebras $(\widehat{\mathfrak{gl}}(m+n),\widehat{\mathfrak{sl}}(n))$ and the coset vertex algebra of the pair of rectangular $W$-algebras $\mathcal{W}^k(\mathfrak{gl}(2m+2n),(2^{m+n}))$ and $\mathcal{W}^{k+m}(\mathfrak{sl}(2n),(2^{n}))$.
\end{abstract}

\textbf{keyword}: Yangian, evaluation map, $W$-algebra, coset

\section{Introduction}
The Yangian $Y_\hbar(\mathfrak{g})$ associated with a finite dimensional simple Lie algebra $\mathfrak{g}$ was introduced by Drinfeld (\cite{D1}, \cite{D2}). The Yangian $Y_\hbar(\mathfrak{g})$ is a quantum group which is a deformation of the current algebra $\mathfrak{g}\otimes\mathbb{C}[z]$. The Yangian has several presentations: the RTT presentation, the current presentation, the Drinfeld $J$ presentation and so on. 

By using the current presentation of the Yangian, we can extend the definition of the Yangian $Y_\hbar(\mathfrak{g})$ to a symmetrizable Kac-Moody Lie algebra $\mathfrak{g}$. The affine Yangian $Y_{\hbar,\ve}(\widehat{\mathfrak{sl}}(n))$ was first introduced by Guay (\cite{Gu2} and \cite{Gu1}). The affine Yangian $Y_{\hbar,\ve}(\widehat{\mathfrak{sl}}(n))$ is a 2-parameter affine Yangian associated with $\widehat{\mathfrak{sl}}(n)$ and a quantum group which is a deformation of the universal enveloping algebra of the universal central extension of $\mathfrak{sl}(n)[u^{\pm1},v]$. Recently, the affine Yangian is applied to the study of a $W$-algebra. The $W$-algebra $\mathcal{W}^k(\mathfrak{g},f)$ is a vertex algebra associated with a finite dimensional reductive Lie algebra $\mathfrak{g}$, a nilpotent element $f\in\mathfrak{g}$ and a complex number $k$. In \cite{U4}, we gave a relationship between the affine Yangian and a rectangular $W$-algebra. The rectangular $W$-algebra $\mathcal{W}^k(\mathfrak{gl}(ln),(l^n))$ is a $W$-algebra associated with $\mathfrak{gl}(ln)$ and a nilpotent element of type $(l^n)$.
The author \cite{U4} gave a surjective homomorphism 
\begin{equation*}
\Phi^n\colon Y_{\hbar,\ve}(\widehat{\mathfrak{sl}}(n))\to\mathcal{U}(\mathcal{W}^k(\mathfrak{gl}(ln),(l^n))),
\end{equation*}
where $\mathcal{U}(\mathcal{W}^k(\mathfrak{gl}(ln),(l^n)))$ is the universal enveloping algebra of $\mathcal{W}^k(\mathfrak{gl}(ln),(l^n))$.

One of the difference between finite Yangians and affine Yangians is the existence of the RTT presentation.
By using the RTT presentation, we find that there exists a natural embedding $\Psi_1^f$ from the finite Yangian associated with $\mathfrak{gl}(n)$ to the finite Yangian associated with $\mathfrak{gl}(n+m)$. In the affine setting, the author \cite{U8} constructed a homomorphism corresponding to $\Psi_1^f$:
\begin{equation*}
\Psi_1\colon Y_{\hbar,\ve}(\widehat{\mathfrak{sl}}(n))\to \widetilde{Y}_{\hbar,\ve}(\widehat{\mathfrak{sl}}(m+n)),
\end{equation*}
where $\widetilde{Y}_{\hbar,\ve}(\widehat{\mathfrak{sl}}(m+n))$ is the degreewise completion of $Y_{\hbar,\ve}(\widehat{\mathfrak{sl}}(m+n))$. 

In this article, we construct a homomorphism 
\begin{equation*}
\Psi_2\colon Y_{\hbar,\ve+n\hbar}(\widehat{\mathfrak{sl}}(m))\to \widetilde{Y}_{\hbar,\ve}(\widehat{\mathfrak{sl}}(m+n)),
\end{equation*}
where $\Psi_2(Y_{\hbar,\ve+n\hbar}(\widehat{\mathfrak{sl}}(m)))$ and $\Psi_1(Y_{\hbar,\ve}(\widehat{\mathfrak{sl}}(n)))$ are commutative with each other. 

One of the applications of homomorphisms $\Psi_1$ and $\Psi_2$ is a relatioship between affine Yangians and $W$-algebras. Brundan-Kleshchev \cite{BK0} gave the parabolic presentation of the finite Yangian of type $A$. By using the parabolic presentation, Brundan-Kleshchev \cite{BK} wrote down  a finite $W$-algebra of type $A$ as a quotient algebra of the shifted Yangian. In affine setting, Crutzig-Diaconescu-Ma \cite{CE} conjectured that an action of an iterated $W$-algebra of type $A$ on the equivariant homology space of the affine Laumon space will be given through an action of an shifted affine Yangian constructed in \cite{FT}. There exists another version of this conjecture which notes the existence of a surjective homomorphism from the shifted affine Yangian to the universal enveloping algebra of a $W$-algebra of type $A$ if we change the definition of the shifted affine Yangian properly. The image of the homomorphism $\Psi_1\otimes\Psi_2$ is corresponding to the Levi subalgebra of the finite Yangian, which is defined by using the parabolic presentation. Moreover, by using $\Psi_1$ and $\Psi_2$, we have shown that there exist a homomorphism from the Levi subalgebras of the affine Yangian to the universal enveloping algebra of a $W$-algebra of type $A$ in the sequence of this article (see \cite{U11} and \cite{U12}). We expect that the homomorphisms $\Psi_1$ and $\Psi_2$ will lead the new presentation of the affine Yangian and be helpful for the resolution of the another version of the Crutzig-Diaconescu-Ma's conjecture.

Another application of  $\Psi_1$ and $\Psi_2$ is a centralizer algebra of $\widehat{\mathfrak{gl}}(n)$ and a rectangular $W$-algebra. For associative algebras $A$ and $B$, we set
\begin{equation*}
C(A,B)=\{x\in A\mid xy=yx\text{ for }y\in B\}.
\end{equation*}
In the finite setting, Olshanskii \cite{Ol} gave a homomorphism from the finite Yangian associated with $\mathfrak{gl}(m)$ to $C(U(\mathfrak{gl}(m+n)),U(\mathfrak{gl}(n)))$. By using the parabolic presentation of Brundan-Kleshchev \cite{BK0}, $C(U(\mathfrak{gl}(m+n)),U(\mathfrak{gl}(n)))$ can be decomposed into the tensor product of the center of $U(\mathfrak{gl}(n))$ and the image of the evaluation map of the Yangian associated with $\mathfrak{gl}(m)$. Moreover, the finite Yangian associated with $\mathfrak{gl}(m)$ can be embedded into the projective limit of this centralizer algebra.

The affine Yangian has a surjective homomorphism called the evaluation map (Guay \cite{Gu1} and Kodera \cite{K1}, \cite{K2}):
\begin{equation*}
\ev^n_{\hbar,\ve}\colon Y_{\hbar,\ve}(\widehat{\mathfrak{sl}}(n))\to U(\widehat{\mathfrak{gl}}(n)).
\end{equation*}
By combining $\ev^{m+n}_{\hbar,\ve}$ and $\Psi_2$, we obtain a homomorphism
\begin{equation*}
\ev^{m+n}_{\hbar,\ve}\circ\Psi_2\colon Y_{\hbar,\ve+n\hbar}(\widehat{\mathfrak{sl}}(m))\to C(U(\widehat{\mathfrak{gl}}(m+n)),U(\widehat{\mathfrak{gl}}(n))).
\end{equation*}
Similarly to finite setting, we expect that the affine Yangian can be embedded into the projective limit of the centralizer algebra $C(U(\widehat{\mathfrak{gl}}(m+n)),U(\widehat{\mathfrak{gl}}(n)))$ through this homomorphism. We also conjecture that $C(U(\widehat{\mathfrak{gl}}(m+n)),U(\widehat{\mathfrak{gl}}(n)))$ is isomorphic to the tensor product of the center of $U(\widehat{\mathfrak{gl}}(n))$ and the image of $\ev^{m+n}_{\hbar,\ve}\circ\Psi_2$.

The similar result holds for rectangular $W$-algebras. There exists a natural embedding from the rectangular $W$-algebra $\mathcal{W}^{k+m}(\mathfrak{gl}(2n),(2^{n}))$ to $\mathcal{W}^k(\mathfrak{gl}(2(m+n),(2^{m+n}))$. 
By combining $\Phi^{m+n}$ and $\Psi_2$, we can obtain a homomorphism
\begin{equation*}
\Phi^{m+n}\circ\Psi_2\colon Y_{\hbar,\ve+n\hbar}(\widehat{\mathfrak{sl}}(m))\to C(\mathcal{U}(\mathcal{W}^k(\mathfrak{gl}(2m+2n),(2^{m+n}))),\mathcal{U}(\mathcal{W}^{k+m}(\mathfrak{gl}(2n),(2^{n})))).
\end{equation*}
Similarly to the evaluation map, we expect that $C(U(\widehat{\mathfrak{gl}}(m+n)),U(\widehat{\mathfrak{gl}}(n)))$ is isomorphic to the tensor product of the center of the universal envelpoping algebra of the rectangular $W$-algebra and the image of $\Phi^{m+n}\circ\Psi_2$.
We also conjecture that we can obtain the similar homomorhism for any $l\geq 3$. For extending to the case that $l\geq3$, we only need to show that we can embed the rectangular $W$-algebra $\mathcal{W}^{k+(l-1)m}(\mathfrak{gl}(ln),(l^{n}))$ to $\mathcal{W}^k(\mathfrak{gl}(lm+ln),(l^{m+n}))$ naturally. 

Kodera-Ueda \cite{KU} gave the meaning to $\Phi^n$ from the perspective of $l$. The meaning is that the coproduct for the affine Yangian corresponds to the parabolic induction for a rectangular $W$-algebra via the homomorphism $\Phi^n$. The result of this article attaches the meaning to $\Phi^n$ from the perspective of $n$. 

We expect that this result can be applicable to the generalization of the Gaiotto-Rapcak's triality.
Gaiotto and Rapcak \cite{GR} introduced a kind of vertex algebras called $Y$-algebras. The $Y$-algebras are related to twisted $N=4$ supersymmetric gauge theories. Gaiotto-Rapcak \cite{GR} conjectured a triality of the isomorphism of $Y$-algebras. Let us consider a $W$-algebra associated with $\mathfrak{sl}(m+n)$ and its nipotent element $f_{n,m}$ of type $(n^1,1^m)$. The nilpotent element $f_{n,m}$ can be decomposed into two nilpotent elments: one is a principal nilpotent element of $\mathfrak{gl}(n)$ and another is a nilpotent element of type $(1^n)$, that is, zero. The $W$-algebra $\mathcal{W}^k(\mathfrak{sl}(m+n),f_{n,m})$ naturally contains the universal affine vertex algebra $V^{k-m-1}(\mathfrak{gl}(m))$, which is corresponding to the latter nilpotent element. It is known that some kinds of $Y$-algebras can be realized as a coset of the pair of  $\mathcal{W}^k(\mathfrak{sl}(m+n),f_{n,m})$ and $V^{k-m-1}(\mathfrak{gl}(m))$ up to Heisenberg algebras. In this case, Creutzig-Linshaw \cite{CR} have proved the triality conjecture. This result is the generalization of the Feigin-Frenkel duality and the coset realization of principal $W$-algebra.

The $Y$-algebras can be interpreted as a truncation of $\mathcal{W}_{1+\infty}$-algebra (\cite{GG}), whose universal enveloping algebra is isomorphic to the affine Yangian of $\widehat{\mathfrak{gl}}(1)$ up to suitable completions (see \cite{AS}, \cite{T} and \cite{MO}). For a vertex algebra $A$ and its vertex subalgebra $B$, let us set the coset vertex algebra of the pair $A$ and $B$
\begin{equation*}
Com(A,B)=\{a\in A\mid |b_{(r)}a=0\text{ for }r\geq0,b\in B\}.
\end{equation*}
The homomorphism $\Phi^{m+n}\circ\Psi_2$ induces the one from the affine Yangian $Y_{\hbar,\ve+n\hbar}(\widehat{\mathfrak{sl}}(m))$ to the universal enveloping algebra of $Com(\mathcal{W}^k(\mathfrak{gl}(2m+2n),(2^{m+n})),\mathcal{W}^k(\mathfrak{sl}(2n),(2^{n})))$. For non-negative integers $n_1$ and $n_2$, we expect that this homomorphism becomes surjective and induces the isomorphism 
\begin{align*}
Com(\mathcal{W}^k(\mathfrak{gl}(2m+2n_1),(2^{m+n_1})),&\mathcal{W}^{k+m}
(\mathfrak{sl}(2n_1),(2^{n_1})))\\
&\simeq Com(\mathcal{W}^k(\mathfrak{gl}(2m+2n_2),(2^{m+n_2})),\mathcal{W}^{k+m}(\mathfrak{sl}(2n_2),(2^{n_2}))),
\end{align*}
which is one of the generalizations of the Gaiotto-Rapcak's triality.

\section{Affine Yangian}
Let us recall the definition of the affine Yangian of type $A$ (Definition~3.2 in \cite{Gu2} and Definition~2.3 in \cite{Gu1}). Hereafter, we sometimes identify $\{0,1,2,\cdots,n-1\}$ with $\mathbb{Z}/n\mathbb{Z}$. Let us set$\{X,Y\}=XY+YX$ and
\begin{equation*}
a_{i,j} =\begin{cases}
2&\text{if } i=j, \\
-1&\text{if }j=i\pm 1,\\
0&\text{otherwise}
	\end{cases}
\end{equation*}
for $i\in\mathbb{Z}/n\mathbb{Z}$.
\begin{Definition}\label{Prop32}
Suppose that $n\geq3$. The affine Yangian $Y_{\hbar,\ve}(\widehat{\mathfrak{sl}}(n))$ is the associative algebra  generated by $X_{i,r}^{+}, X_{i,r}^{-}, H_{i,r}$ $(i \in \{0,1,\cdots, n-1\}, r = 0,1)$ subject to the following defining relations:
\begin{gather}
[H_{i,r}, H_{j,s}] = 0,\label{Eq2.1}\\
[X_{i,0}^{+}, X_{j,0}^{-}] = \delta_{i,j} H_{i, 0},\label{Eq2.2}\\
[X_{i,1}^{+}, X_{j,0}^{-}] = \delta_{i,j} H_{i, 1} = [X_{i,0}^{+}, X_{j,1}^{-}],\label{Eq2.3}\\
[H_{i,0}, X_{j,r}^{\pm}] = \pm a_{i,j} X_{j,r}^{\pm},\label{Eq2.4}\\
[\tilde{H}_{i,1}, X_{j,0}^{\pm}] = \pm a_{i,j}\left(X_{j,1}^{\pm}\right),\text{ if }(i,j)\neq(0,n-1),(n-1,0),\label{Eq2.5}\\
[\tilde{H}_{0,1}, X_{n-1,0}^{\pm}] = \mp \left(X_{n-1,1}^{\pm}+(\ve+\dfrac{n}{2}\hbar) X_{n-1, 0}^{\pm}\right),\label{Eq2.6}\\
[\tilde{H}_{n-1,1}, X_{0,0}^{\pm}] = \mp \left(X_{0,1}^{\pm}-(\ve+\dfrac{n}{2}\hbar) X_{0, 0}^{\pm}\right),\label{Eq2.7}\\
[X_{i, 1}^{\pm}, X_{j, 0}^{\pm}] - [X_{i, 0}^{\pm}, X_{j, 1}^{\pm}] = \pm a_{ij}\dfrac{\hbar}{2} \{X_{i, 0}^{\pm}, X_{j, 0}^{\pm}\}\text{ if }(i,j)\neq(0,n-1),(n-1,0),\label{Eq2.8}\\
[X_{0, 1}^{\pm}, X_{n-1, 0}^{\pm}] - [X_{0, 0}^{\pm}, X_{n-1, 1}^{\pm}]= \mp\dfrac{\hbar}{2} \{X_{0, 0}^{\pm}, X_{n-1, 0}^{\pm}\} + (\ve+\dfrac{n}{2}\hbar) [X_{0, 0}^{\pm}, X_{n-1, 0}^{\pm}],\label{Eq2.9}\\
(\ad X_{i,0}^{\pm})^{1+|a_{i,j}|} (X_{j,0}^{\pm})= 0 \ \text{ if }i \neq j, \label{Eq2.10}
\end{gather}
where $\widetilde{H}_{i,1}=H_{i,1}-\dfrac{\hbar}{2}H_{i,0}^2$.
\end{Definition}
\begin{Remark}
Definition~\ref{Prop32} is different from Definition~3.2 in \cite{Gu2} and Definition~2.3 in \cite{Gu1}. Guay-Nakajima-Wendlandt \cite{GNW} gave the minimalistic presentation of the affine Yangian. Definition~\ref{Prop32} can be derived from the minimalistic presentation (see Section 2 in \cite{U8}).
\end{Remark}
By using the defining relations of the affine Yangian, we find the following relations (see Section 2 in \cite{U8}):
\begin{gather}
[X^\pm_{i,r},X^\pm_{j,s}]=0\text{ if }|i-j|>1,\label{gather1}\\
[X^\pm_{i,1},[X^\pm_{i,0},X^\pm_{j+1,r}]]+[X^\pm_{i,0},[X^\pm_{i,1},X^\pm_{j+1,r}]]=0.\label{gather2}
\end{gather}
By the definition of the affine Yangian $Y_{\hbar,\ve}(\widehat{\mathfrak{sl}}(n))$, we find that there exists a natural homomorphism from the universal enveloping algebra of $\widehat{\mathfrak{sl}}(n)$ to $Y_{\hbar,\ve}(\widehat{\mathfrak{sl}}(n))$. In order to simplify the notation, we denote the image of $x\in U(\widehat{\mathfrak{sl}}(n))$ via this homomorphism by $x$.

We take one completion of $Y_{\hbar,\ve}(\widehat{\mathfrak{sl}}(n))$. We set the degree of $Y_{\hbar,\ve}(\widehat{\mathfrak{sl}}(n))$ by
\begin{equation*}
\text{deg}(H_{i,r})=0,\ \text{deg}(X^\pm_{i,r})=\begin{cases}
\pm1&\text{ if }i=0,\\
0&\text{ if }i\neq0.
\end{cases}
\end{equation*}
We denote the standard degreewise completion of $Y_{\hbar,\ve}(\widehat{\mathfrak{sl}}(n))$ by $\widetilde{Y}_{\hbar,\ve}(\widehat{\mathfrak{sl}}(n))$ (see Section 1.3 in \cite{MNT} and Section A.2 in \cite{A1}). Let us set $A_i\in\widetilde{Y}_{\hbar,\ve}(\widehat{\mathfrak{sl}}(n))$ as
\begin{align*}
A_i&=\dfrac{\hbar}{2}\sum_{\substack{s\geq0\\u>v}}\limits E_{u,v}t^{-s}[E_{i,i},E_{v,u}t^s]+\dfrac{\hbar}{2}\sum_{\substack{s\geq0\\u<v}}\limits E_{u,v}t^{-s-1}[E_{i,i},E_{v,u}t^{s+1}]\\
&=\dfrac{\hbar}{2}\sum_{\substack{s\geq0\\u>i}}\limits E_{u,i}t^{-s}E_{i,u}t^s-\dfrac{\hbar}{2}\sum_{\substack{s\geq0\\i>v}}\limits E_{i,v}t^{-s}E_{v,i}t^s\\
&\quad+\dfrac{\hbar}{2}\sum_{\substack{s\geq0\\u<i}}\limits E_{u,i}t^{-s-1}E_{i,u}t^{s+1}-\dfrac{\hbar}{2}\sum_{\substack{s\geq0\\i<v}}\limits E_{i,v}t^{-s-1}E_{v,i}t^{s+1},
\end{align*}
where $E_{i,j}$ is a matrix unit whose $(a,b)$ component is $\delta_{a,i}\delta_{b,j}$.
Similarly to Section~3 in \cite{GNW}, we define
\begin{align*}
J(h_i)&=\widetilde{H}_{i,1}+A_i-A_{i+1}\in \widetilde{Y}_{\hbar,\ve}(\widehat{\mathfrak{sl}}(n)).
\end{align*}
We also set $J(x^\pm_i)=\pm\dfrac{1}{2}[J(h_i),x^\pm_i]$.

Guay-Nakajima-Wendlandt \cite{GNW} defined the automorphism of $Y_{\hbar,\ve}(\widehat{\mathfrak{sl}}(n))$ by
\begin{equation*}
\tau_i=\exp(\ad(x^+_{i,0}))\exp(-\ad(x^-_{i,0}))\exp(\ad(x^+_{i,0})).
\end{equation*}
Let $\alpha$ be a positive real root. There is an element $w$ of the Weyl group of $\widehat{\mathfrak{sl}}(n)$ and a simple root $\alpha_j$ such that $\alpha=w\alpha_j$. Then we define a corresponding root vector by
\begin{equation*}
x^\pm_\alpha=\tau_{i_1}\tau_{i_2}\cdots\tau_{i_{p-1}}(x^\pm_{j}),
\end{equation*}
where $w =s_{i_1}s_{i_2}\cdots s_{i_{p-1}}$ is a reduced expression of $w$.
We can define $J(x^\pm_\alpha)$ as
\begin{equation*}
J(x^\pm_\alpha)=\tau_{i_1}\tau_{i_2}\cdots\tau_{i_{p-1}}J(x^\pm_{j}).
\end{equation*} 
\begin{Lemma}[(3.14) and Proposition 3.21 in \cite{GNW}]\label{J}
\begin{enumerate}
\item The following relations hold:
\begin{gather}
[J(h_i),X^\pm_{j,0}]=\pm a_{ij}J(x^\pm_j)\text{ if }(i,j)\neq(0,n-1),(n-1,0),\\
[J(x^\pm_i), X_{j, 0}^{\pm}]=[X_{i, 0}^{\pm}, J(x^\pm_j)]\text{ if }(i,j)\neq(0,n-1),(n-1,0),\\
[J(x^\pm_i),X^\pm_{j,0}]=0\text{ if }|i-j|>1.
\end{gather}
\item
There exists $c_{\alpha,i}\in\mathbb{C}$ satisfying that
\begin{equation*}
[J(h_i),x^\pm_\alpha]=\pm(\alpha_i,\alpha)J(x_\alpha^\pm)\pm c_{\alpha,i}x_\alpha^\pm.
\end{equation*}
\end{enumerate}
\end{Lemma}
\section{A homomorphism from the affine Yangian $Y_{\hbar,\ve+\hbar}(\widehat{\mathfrak{sl}}(n))$ to the affine Yangian $Y_{\hbar,\ve}(\widehat{\mathfrak{sl}}(n+1))$}
In this section, we will construct a homomorphism from the affine Yangian $Y_{\hbar,\ve}(\widehat{\mathfrak{sl}}(n))$ to the degreewise completion of the affine Yangian $Y_{\hbar,\ve}(\widehat{\mathfrak{sl}}(n+1))$, which is different from the one given in \cite{U8}.
\begin{Theorem}\label{Main}
There exists an algebra homomorphism
\begin{equation*}
\Psi\colon Y_{\hbar,\ve+\hbar}(\widehat{\mathfrak{sl}}(n))\to \widetilde{Y}_{\hbar,\ve}(\widehat{\mathfrak{sl}}(n+1))
\end{equation*}
determined by
\begin{gather*}
\Psi(H_{i,0})=\begin{cases}
H_{0,0}+H_{1,0}&\text{ if }i=0,\\
H_{i+1,0}&\text{ if }i\neq 0,
\end{cases}\\
\Psi(X^+_{i,0})=\begin{cases}
E_{n+1,2}t&\text{ if }i=0,\\
E_{i+1,i+2}&\text{ if }i\neq 0,
\end{cases}\ 
\Psi(X^-_{i,0})=\begin{cases}
E_{2,n+1}t^{-1}&\text{ if }i=0,\\
E_{i+2,i+1}&\text{ if }i\neq 0,
\end{cases}
\end{gather*}
and
\begin{align*}
\Psi(H_{i,1})&= H_{i+1,1}+\hbar\displaystyle\sum_{s \geq 0} \limits E_{1,i+1}t^{-s-1}E_{i+1,1}t^{s+1} -\hbar\displaystyle\sum_{s \geq 0}\limits E_{1,i+2}t^{-s-1} E_{i+2,1}t^{s+1},\\
\Psi(X^+_{i,1})&=X^+_{i+1,1}+\hbar\displaystyle\sum_{s \geq 0}\limits E_{1,i+2}t^{-s-1} E_{i+1,1}t^{s+1},\\
\Psi(X^-_{i,1})&=X^-_{i+1,1}+\hbar\displaystyle\sum_{s \geq 0}\limits E_{1,i+1}t^{-s-1} E_{i+2,1}t^{s+1}
\end{align*}
for $i\neq 0$. In particular, we have
\begin{align*}
\Psi(\widetilde{H}_{i,1})&= \widetilde{H}_{i+1,1}+\hbar\displaystyle\sum_{s \geq 0} \limits E_{1,i+1}t^{-s-1}E_{i+1,1}t^{s+1} -\hbar\displaystyle\sum_{s \geq 0}\limits E_{1,i+2}t^{-s-1} E_{i+2,1}t^{s+1}\text{ for }i\neq 0.
\end{align*}
\end{Theorem}
\begin{Remark}
In \cite{U7}, we gave a homomorphism from the affine Yangian to the universal enveloping algebra of a non-rectangualar $W$-algebra of type $A$ by constructing one kind of the coproduct for the extended affine Yangian. However, the meaning of this coproduct is not clear. One of the motivation of the construction of $\Psi$ is to give the meaning to this coproduct. Based on $B_i^\pm$ defined in Theorem 3.17 of \cite{U7}, we can expect that there exists a homomorphism from the affine Yangian associated with $\widehat{\mathfrak{sl}}(n)$ to the one associated with $\widehat{\mathfrak{sl}}(n+1)$ whose form is
\begin{align*}
\Psi(X^+_{i,1})&=X^+_{i+1,1}+b\hbar\displaystyle\sum_{s \geq 0}\limits E_{1,i+2}t^{-s-a} E_{i+1,1}t^{s+a},\\
\Psi(X^-_{i,1})&=X^-_{i+1,1}+\hbar\displaystyle\sum_{s \geq 0}\limits E_{1,i+1}t^{-s-a} E_{i+2,1}t^{s+a}
\end{align*}
for $i\neq 0$ and some $a\in\mathbb{Z},b\in\{\pm1\}$. In \cite{U11}, by using the homomorphism $\Psi$, we gave the another proof to the main theorem of \cite{U7}. This result is one of the interpretations of the coproduct for the extended affine Yangian.
\end{Remark}
\begin{Corollary}
The following relations hold:
\begin{align*}
\Psi(H_{0,1})&= H_{0,1}+H_{1,1}+\hbar H_{0,0}H_{1,0}+\dfrac{\hbar}{2}H_{0,0}\\
&\quad-\hbar\displaystyle\sum_{s \geq 0} \limits E_{1,2}t^{-s-1}E_{2,1}t^{s+1}+\hbar\displaystyle\sum_{s \geq 0} \limits E_{1,n+1}t^{-s-1}E_{n+1,1}t^{s+1},\\
\Psi(X^+_{0,1})&=[X^+_{0,0},X^+_{1,1}]+\hbar\displaystyle\sum_{s \geq 0} \limits E_{1,2}t^{-s-1}E_{n+1,1}t^{s+2},\\
\Psi(X^-_{i,1})&=[X^-_{1,0},X^-_{0,1}]+\hbar\displaystyle\sum_{s \geq 0} \limits E_{1,n+1}t^{-s-2}E_{2,1}t^{s+1}.
\end{align*}
In particular, we obtain
\begin{align*}
\Psi(\widetilde{H}_{0,1})&=\widetilde{H}_{0,1}+\widetilde{H}_{1,1}+\dfrac{\hbar}{2}H_{0,0}\\
&\quad-\hbar\displaystyle\sum_{s \geq 0} \limits E_{1,2}t^{-s-1}E_{2,1}t^{s+1}+\hbar\displaystyle\sum_{s \geq 0} \limits E_{1,n+1}t^{-s-1}E_{n+1,1}t^{s+1}. 
\end{align*}
\end{Corollary}
\begin{proof}
First, we show the relation for $\Psi(X^+_{0,1})$. By \eqref{Eq2.5} and the definition of $\Psi(\widetilde{H}_{1,1})$, we have
\begin{align}
\Psi(X^+_{0,1})&=-[\widetilde{H}_{2,1},E_{n+1,2}t]\nonumber\\
&\quad-[\hbar\displaystyle\sum_{s \geq 0} \limits E_{1,2}t^{-s-1}E_{2,1}t^{s+1},E_{n+1,2}t]+[\hbar\displaystyle\sum_{s \geq 0}\limits E_{1,3}t^{-s-1} E_{3,1}t^{s+1},E_{n+1,2}t]\nonumber\\
&=-[\widetilde{H}_{2,1},E_{n+1,2}t]+\hbar\displaystyle\sum_{s \geq 0} \limits E_{1,2}t^{-s-1}E_{n+1,1}t^{s+2}.\label{910}
\end{align}
By \eqref{Eq2.5}, we obtain
\begin{align*}
-[\widetilde{H}_{2,1},E_{n+1,2}t]&=-[\widetilde{H}_{2,1},[X^+_{0,0},X^+_{1,0}]]=[X^+_{0,0},X^+_{1,1}].
\end{align*}
Thus, we have proven the relation for $\Psi(X^+_{0,1})$.
Similarly, we can obtain the relation for $\Psi(X^-_{0,1})$.

Next, we show the relation for $\Psi(H_{0,1})$. By \eqref{Eq2.3} and the relation for $\Psi(X^+_{0,1})$, we obtain
\begin{align}
\Psi(H_{0,1})&=[[X^+_{0,0},X^+_{1,1}],[X^-_{1,0},X^-_{0,0}]]+[\hbar\displaystyle\sum_{s \geq 0} \limits E_{1,2}t^{-s-1}E_{n+1,1}t^{s+2},E_{2,n+1}t^{-1}]\nonumber\\
&=[[X^+_{0,0},X^+_{1,1}],[X^-_{1,0},X^-_{0,0}]]\nonumber\\
&\quad-\hbar\displaystyle\sum_{s \geq 0} \limits E_{1,2}t^{-s-1}E_{2,1}t^{s+1}+\hbar\displaystyle\sum_{s \geq 0} \limits E_{1,n+1}t^{-s-2}E_{n+1,1}t^{s+2}.\label{911}
\end{align}
By \eqref{Eq2.1}-\eqref{Eq2.5}, we can rewrite the first term of \eqref{911} as follows:
\begin{align}
[[X^+_{0,0},X^+_{1,1}],[X^-_{1,0},X^-_{0,0}]]&=[[X^+_{0,0},H_{1,1}],X^-_{0,0}]+[X^-_{1,0},[H_{0,0},X^+_{1,1}]]\nonumber\\
&=-[[\widetilde{H}_{1,1}+\dfrac{\hbar}{2}H_{1,0}^2,X^+_{0,0}],X^-_{0,0}]+H_{1,1}\nonumber\\
&=H_{0,1}+H_{1,1}+\dfrac{\hbar}{2}[\{H_{1,0},X^+_{0,0}\},X^-_{0,0}]\nonumber\\
&=H_{0,1}+H_{1,1}+\hbar H_{0,0}H_{1,0}+\dfrac{\hbar}{2}\{X^-_{0,0},X^+_{0,0}\}\nonumber\\
&=H_{0,1}+H_{1,1}+\hbar H_{0,0}H_{1,0}+\hbar X^-_{0,0}X^+_{0,0}+\dfrac{\hbar}{2}H_{0,0}.\label{911-1}
\end{align}
By applying \eqref{911-1} to \eqref{911}, we obtain the relation for $\Psi(H_{0,1})$.
\end{proof}
\begin{proof}[The proof of Theorem~\ref{Main}]
The proof of Theorem~\ref{Main} is similar to Theorem~3.1 in \cite{U8}. In this article, we will show the compatibility with \eqref{Eq2.1} and \eqref{Eq2.9}. The other cases can be proven in a similar way to \cite{U8}.
\subsection{Compatibility of \eqref{Eq2.9}}
We only show the $+$ case. The $-$ case can be proven in the same way. By the definition of $\Psi$, we have
\begin{align*}
[\Psi(X_{n-1, 1}^{+}), \Psi(X_{0, 0}^{+})]
&=[X^+_{n,1},[X^+_{0,0},X^+_{1,0}]]+[\hbar\displaystyle\sum_{s \geq 0}\limits E_{1,n+1}t^{-s-1} E_{n,1}t^{s+1},E_{n+1,2}t]\\
&=[X^+_{n,1},[X^+_{0,0},X^+_{1,0}]]+\hbar\displaystyle\sum_{s \geq 0}\limits E_{1,2}t^{-s} E_{n,1}t^{s+1}
\end{align*}
and
\begin{align*}
[\Psi(X_{0, 1}^{+}), \Psi(X_{n-1, 0}^{+})]
&=[[X^+_{0,0},X^+_{1,1}],X^+_{n,0}]+[\hbar\displaystyle\sum_{s \geq 0} \limits E_{1,2}t^{-s-1}E_{n+1,1}t^{s+2},E_{n,n+1}]\\
&=[[X^+_{0,0},X^+_{1,1}],X^+_{n,0}]-\hbar\displaystyle\sum_{s \geq 0} \limits E_{1,2}t^{-s-1}E_{n,1}t^{s+2}.
\end{align*}
Then, by a direct computation, we obtain
\begin{align}
&\quad[\Psi(X_{0, 1}^{\pm}), \Psi(X_{n-1, 0}^{\pm})]-[\Psi(X_{n-1, 1}^{+}), \Psi(X_{0, 0}^{+})]\nonumber\\
&=[[X^+_{0,0},X^+_{1,1}],X^+_{n,0}]+[X^+_{n,1},[X^+_{0,0},X^+_{1,0}]]+\hbar E_{1,2}E_{n,1}t.\label{912}
\end{align}
By \eqref{Eq2.8}, \eqref{gather1} and \eqref{Eq2.9}, we obtain
\begin{align}
&\quad[X^+_{n,1},[X^+_{0,0},X^+_{1,0}]]+[[X^+_{0,0},X^+_{1,1}],X^+_{n,0}]\nonumber\\
&=[X^+_{n,1},[X^+_{0,0},X^+_{1,0}]]+[[X^+_{0,1},X^+_{1,0}],X^+_{n,0}]+[\dfrac{\hbar}{2}\{X^+_{0,0},X^+_{1,0}\},X^+_{n,0}]\nonumber\\
&=[[X^+_{n,1},X^+_{0,0}]+[X^+_{0,1},X^+_{n,0}],X^+_{1,0}]+\dfrac{\hbar}{2}\{[X^+_{0,0},X^+_{n,0}],X^+_{1,0}\}\nonumber\\
&=[-\dfrac{\hbar}{2} \{X_{0, 0}^{+}, X_{n, 0}^{+}\} + (\ve+\dfrac{n+1}{2}\hbar) [X_{0, 0}^{+}, X_{n, 0}^{+}],X^+_{1,0}]+\dfrac{\hbar}{2}\{[X^+_{0,0},X^+_{n,0}],X^+_{1,0}\}\nonumber\\
&=-\dfrac{\hbar}{2} \{E_{n+1,2}t, E_{n,n+1}\}+ (\ve+\dfrac{n+1}{2}\hbar) [E_{n+1,2}t, E_{n,n+1}]-\dfrac{\hbar}{2}\{E_{n,1}t,E_{1,2}\}.\label{912-1}
\end{align}
By applying \eqref{912-1} to \eqref{912}, we obtain
\begin{align*}
&\quad[\Psi(X_{0, 1}^{\pm}), \Psi(X_{n-1, 0}^{\pm})]-[\Psi(X_{n-1, 1}^{+}), \Psi(X_{0, 0}^{+})]\nonumber\\
&=-\dfrac{\hbar}{2} \{E_{n+1,2}t, E_{n,n+1}\}+ (\ve+\dfrac{n+1}{2}\hbar+\dfrac{\hbar}{2}) [E_{n+1,2}t, E_{n,n+1}].
\end{align*}
Thus, we have proven the compatibility with \eqref{Eq2.9}.
\subsection{The compatibility with \eqref{Eq2.1}}
By the definition of $\Psi$, it is enough to show the relation $[\Psi(\widetilde{H}_{i,1}),\Psi(\widetilde{H}_{j,1})]=0$. We only show the case that $i,j\neq 0$. The other cases can be proven in a similar way. Let us set
\begin{equation*}
R_i=\hbar\displaystyle\sum_{s \geq 0} \limits E_{1,1+i}t^{-s-1}E_{1+i,1}t^{s+1}.
\end{equation*}
By the definition of $J(h_i)$ and $\Psi$, we have
\begin{align}
&\quad[\Psi(\widetilde{H}_{i,1}),\Psi(\widetilde{H}_{j,1})]=[\widetilde{H}_{i,1}+R_i-R_{i+1},\widetilde{H}_{j,1}+R_j-R_{j+1}]\nonumber\\
&=0+[\widetilde{H}_{1+i,1},R_j-R_{j+1}]+[R_i-R_{i+1},\widetilde{H}_{1+j,1}]+[R_i-R_{i+1},R_j,R_{j+1}]\nonumber\\
&=[J(h_{1+i})-A_{1+i}+A_{2+i},R_j-R_{j+1}]\nonumber\\
&\quad+[R_i-R_{i+1},J(h_{1+j})-A_{1+j}+A_{2+j}]+[R_i-R_{i+1},R_j-R_{j+1}]\nonumber\\
&=[-A_{1+i}+A_{2+i},R_j-R_{j+1}]+[R_i-R_{i+1},-A_{1+j}+A_{2+j}]+[R_i-R_{i+1},R_j-R_{j+1}],\label{9133}
\end{align}
where the last equality is due to Lemma~\ref{J}. Thus, it is enough to show the relation
\begin{equation*}
-[A_{1+i},R_j]+[A_{1+j},R_i]+[R_i,R_j]=0.\label{913}
\end{equation*}
We will compute all terms of the right hand side of \eqref{913}.
By a direct computation, we obtain
\begin{align}
&\quad[R_i,R_j]\nonumber\\
&=\hbar^2\displaystyle\sum_{s,v \geq 0}\limits E_{1,1+i}t^{-s-1}(E_{1+i,1+j}t^{s-v})E_{1+j,1}t^{v+1}-\hbar^2\displaystyle\sum_{s,v \geq 0}\limits E_{1,1+j}t^{-v-1}(E_{1+j,1+i}t^{v-s})E_{1+i,1}t^{s+1}.\label{550-0}
\end{align}
By the definition of $A_i$, we obtain
\begin{align}
[A_{1+i},R_j]
&=[\dfrac{\hbar}{2}\sum_{\substack{s\geq0\\u>1+i}}\limits E_{u,1+i}t^{-s}E_{1+i,u}t^s,R_j]-[\dfrac{\hbar}{2}\sum_{\substack{s\geq0\\1+i>u}}\limits E_{1+i,u}t^{-s}E_{u,1+i}t^s,R_j]\nonumber\\
&\quad+[\dfrac{\hbar}{2}\sum_{\substack{s\geq0\\u<1+i}}\limits E_{u,1+i}t^{-s-1}E_{1+i,u}t^{s+1},R_j]-[\dfrac{\hbar}{2}\sum_{\substack{s\geq0\\1+i<u}}\limits E_{1+i,u}t^{-s-1}E_{u,1+i}t^{s+1},R_j].\label{9114}
\end{align}
We compute the right hand  side of \eqref{9114}. By \eqref{AQ-1} and \eqref{AQ-2}, we obtain
\begin{align}
\eqref{9114}_1
&=-\dfrac{\hbar^2}{2}\delta_{i,j}\sum_{\substack{s,v\geq0\\u>1+i}}\limits E_{u,1+i}t^{-s}E_{1,u}t^{s-v-1}E_{1+j,1}t^{v+1}\nonumber\\
&\quad-\dfrac{\hbar^2}{2}\delta(j>i)\sum_{\substack{s,v\geq0}}\limits E_{1,1+i}t^{-s-v-1}E_{1+i,1+j}t^sE_{1+j,1}t^{v+1}\nonumber\\
&\quad+\dfrac{\hbar^2}{2}\delta(j>i)\sum_{\substack{s,v\geq0}}\limits E_{1,1+j}t^{-v-1}E_{1+j,1+i}t^{-s}E_{1+i,1}t^{s+v+1}\nonumber\\
&\quad+\dfrac{\hbar^2}{2}\delta_{i,j}\sum_{\substack{s,v\geq0\\u>1+i}}\limits E_{1,1+j}t^{-v-1}E_{u,1}t^{v-s+1}E_{1+i,u}t^s,\label{550-1}\\
\eqref{9114}_2&=\dfrac{\hbar^2}{2}\delta(i>j)\sum_{\substack{s,v\geq0}}\limits E_{1+i,1+j}t^{-s}E_{1,1+i}t^{s-v-1}E_{1+j,1}t^{v+1}\nonumber\\
&\quad-\dfrac{\hbar^2}{2}\sum_{\substack{s,v\geq0}}\limits E_{1+i,1+j}t^{-s-v-1}E_{1,1+i}t^sE_{1+j,1}t^{v+1}\nonumber\\
&\quad+\dfrac{\hbar^2}{2}\delta_{i,j}\sum_{\substack{s,v\geq0\\1+i>u}}\limits E_{1,u}t^{-s-v-1}E_{u,1+i}t^sE_{1+j,1}t^{v+1}\nonumber\\
&\quad-\dfrac{\hbar^2}{2}\delta_{i,j}\sum_{\substack{s,v\geq0\\1+i>u}}\limits E_{1,1+j}t^{-v-1}E_{1+i,u}t^{-s}E_{u,1}t^{s+v+1}\nonumber\\
&\quad+\dfrac{\hbar^2}{2}\sum_{\substack{s,v\geq0}}\limits E_{1,1+j}t^{-v-1}E_{1+i,1}t^{-s}E_{1+j,1+i}t^{s+v+1}\nonumber\\
&\quad-\dfrac{\hbar^2}{2}\delta(i>j)\sum_{\substack{s,v\geq0}}\limits E_{1,1+j}t^{-v-1}E_{1+i,1}t^{v-s+1}E_{1+j,1+i}t^s,\label{550-2}\\
\eqref{9114}_3&=\dfrac{\hbar^2}{2}\sum_{\substack{s,v\geq0}}\limits E_{1,1+i}t^{-s-1}E_{1+i,1+j}t^{s-v}E_{1+j,1}t^{v+1}\nonumber\\
&\quad-\dfrac{\hbar}{2}\delta_{i,j}\sum_{\substack{s,v\geq0}}\limits E_{1,1+i}t^{-s-1}E_{1,1}t^{s-v}E_{1+j,1}t^{v+1}\nonumber\\
&\quad-\dfrac{\hbar^2}{2}\delta(j<i)\sum_{\substack{s,v\geq0}}\limits E_{1,1+i}t^{-s-v-2}E_{1+i,1+j}t^{s+1}E_{1+j,1}t^{v+1}\nonumber\\
&\quad+\dfrac{\hbar^2}{2}\delta(j<i)\sum_{\substack{s,v\geq0}}\limits E_{1,1+j}t^{-v-1}E_{1+j,1+i}t^{-s-1}E_{1+i,1}t^{s+v+2}\nonumber\\
&\quad+\dfrac{\hbar^2}{2}\sum_{\substack{s,v\geq0\\u<1+i}}\limits E_{1,1+j}t^{-v-1}E_{u,1}t^{v-s}E_{1+i,u}t^{s+1}\nonumber\\
&\quad-\dfrac{\hbar^2}{2}\delta_{i,j}\sum_{\substack{s,v\geq0}}\limits E_{1,1+j}t^{-v-1}E_{1+j,1+i}t^{v-s}E_{1+i,1}t^{s+1},\label{550-3}\\
\eqref{9114}_4&=\dfrac{\hbar^2}{2}\delta(i<j)\sum_{\substack{s,v\geq0}}\limits E_{1+i,1+j}t^{-s-1}E_{1,1+i}t^{s-v}E_{1+j,1}t^{v+1}\nonumber\\
&\quad+\dfrac{\hbar^2}{2}\delta_{i,j}\sum_{\substack{s,v\geq0\\1+i<u}}\limits E_{1,u}t^{-s-v-1}E_{u,1+i}t^sE_{1+j,1}t^{v+1}\nonumber\\
&\quad-\dfrac{\hbar^2}{2}\delta_{i,j}\sum_{\substack{s,v\geq0\\1+i<u}}\limits E_{1,1+j}t^{-v-1}E_{1+i,u}t^{-s-1}E_{u,1}t^{s+v+2}\nonumber\\
&\quad-\dfrac{\hbar^2}{2}\delta(i<j)\sum_{\substack{s,v\geq0}}\limits E_{1,1+j}t^{-v-1}E_{1+i,1}t^{v-s}E_{1+j,1+i}t^{s+1}.\label{550-4}
\end{align}
Considering the sum $-[A_{1+i},R_j]+[A_{1+j},R_i]$, the terms containg $\delta_{i,j}$ in \eqref{550-1}-\eqref{550-4} vanish each other. Hereafter, in order to simplify the computation, we will denote the $i$-th term of the right hand side of the equation $(\cdot)$ by $(\cdot)_i$.
We divide the following terms into two piecies:
\begin{align}
\eqref{550-2}_1&=\dfrac{\hbar^2}{2}\delta(i>j)\sum_{\substack{s,v\geq0}}\limits E_{1+i,1+j}t^{-s}E_{1,1+i}t^{-v-1}E_{1+j,1}t^{s+v+1}\nonumber\\
&\quad+\dfrac{\hbar^2}{2}\delta(i>j)\sum_{\substack{s,v\geq0}}\limits E_{1+i,1+j}t^{-s-v-1}E_{1,1+i}t^{s}E_{1+j,1}t^{v+1},\label{5502-1}\\
\eqref{550-2}_4&=-\dfrac{\hbar^2}{2}\delta(i>j)\sum_{\substack{s,v\geq0}}\limits E_{1,1+j}t^{-v-1}E_{1+i,1}t^{-s}E_{1+j,1+i}t^{s+v+1}\nonumber\\
&\quad-\dfrac{\hbar^2}{2}\delta(i>j)\sum_{\substack{s,v\geq0}}\limits E_{1,1+j}t^{-v-s-1}E_{1+i,1}t^{v+1}E_{1+j,1+i}t^s,\label{5502-4}\\
\eqref{550-4}_1&=\dfrac{\hbar^2}{2}\delta(i<j)\sum_{\substack{s,v\geq0}}\limits E_{1+i,1+j}t^{-s-1}E_{1,1+i}t^{-v-1}E_{1+j,1}t^{s+2+v}\nonumber\\
&\quad+\dfrac{\hbar^2}{2}\delta(i<j)\sum_{\substack{s,v\geq0}}\limits E_{1+i,1+j}t^{-s-1-v}E_{1,1+i}t^{s}E_{1+j,1}t^{v+1},\label{5504-1}\\
\eqref{550-4}_4&=-\dfrac{\hbar^2}{2}\delta(i<j)\sum_{\substack{s,v\geq0}}\limits E_{1,1+j}t^{-v-s-2}E_{1+i,1}t^{v+1}E_{1+j,i}t^{s+1}\nonumber\\
&\quad-\dfrac{\hbar^2}{2}\delta(i<j)\sum_{\substack{s,v\geq0}}\limits E_{1,1+j}t^{-v-1}E_{1+i,1}t^{-s}E_{1+j,1+i}t^{v+s+1}.\label{5504-4}
\end{align}
In this proof, we denote, we denote the result of substituting $i=k$ and $j=l$ into the equation $(\cdot)$ as $(\cdot)_{k,l}$. By a direct computation, we obtain
\begin{align}
\eqref{550-1}_{2,i,j}-\eqref{5502-4}_{2,j,i}&=-\dfrac{\hbar^2}{2}\delta(j>i)\sum_{\substack{s\geq0}}\limits(s+1)E_{1,1+i}t^{-s-1}E_{1+i,1}t^{s+1},\label{55-1}\\
\eqref{550-1}_{3,i,j}-\eqref{5502-1}_{1,j,i}&=\dfrac{\hbar^2}{2}\delta(j>i)\sum_{\substack{s\geq0}}\limits (s+1)E_{1,1+i}t^{-s-1}E_{1+i,1}t^{s+1},\label{55-2}\\
\eqref{5502-1}_{2,i,j}+\eqref{5504-1}_{2,i,j}&=\dfrac{\hbar^2}{2}\delta(i\neq j)\sum_{\substack{s,v\geq0}}\limits E_{1+i,1+j}t^{-s-v-1}E_{1,1+i}t^{s}E_{1+j,1}t^{v+1},\label{55-3}\\
\eqref{5502-4}_{1,i,j}+\eqref{5504-4}_{2,i,j}&=-\dfrac{\hbar^2}{2}\delta(i\neq j)\sum_{\substack{s,v\geq0}}\limits E_{1,1+j}t^{-v-1}E_{1+i,1}t^{-s}E_{1+j,1+i}t^{s+v+1},\label{55-4}\\
\eqref{550-3}_{3,i,j}-\eqref{5504-4}_{1,j,i}&=-\dfrac{\hbar^2}{2}\delta(j<i)\sum_{\substack{s\geq0}}\limits(s+1)E_{1,1+i}t^{-s-2}E_{1+i,1}t^{s+2},\label{55-5}\\
\eqref{550-3}_{4,i,j}-\eqref{5504-1}_{1,j,i}&=\dfrac{\hbar^2}{2}\delta(j<i)\sum_{\substack{s\geq0}}\limits E_{1,1+i}t^{-s-2}E_{1+i,1}t^{s+2}.\label{55-6}
\end{align}
Since $\eqref{55-1}+\eqref{55-2}=0$, $\eqref{55-5}+\eqref{55-6}=0$ and
\begin{align*}
\eqref{550-2}_2+\eqref{55-3}&=\dfrac{\hbar^2}{2}\delta_{i,j}\sum_{\substack{s,v\geq0}}\limits E_{1+i,1+j}t^{-s-v-1}E_{1,1+i}t^{s}E_{1+j,1}t^{v+1},\\
\eqref{550-2}_5+\eqref{55-4}&=-\dfrac{\hbar^2}{2}\delta_{i,j}\sum_{\substack{s,v\geq0}}\limits E_{1,1+j}t^{-v-1}E_{1+i,1}t^{-s}E_{1+j,1+i}t^{s+v+1}
\end{align*}
hold by a direct computation, we have
\begin{align}
-[A_{1+i},R_j]+[A_{1+j},R_i]&=-\eqref{550-3}_{1,i,j}+\eqref{550-3}_{6,j,i}+\eqref{550-3}_{1,j,i}-\eqref{550-3}_{6,i,j}.
\end{align}
Since we obtain
\begin{align*}
\eqref{550-0}_1-\eqref{550-3}_{1,i,j}+\eqref{550-3}_{6,j,i}&=0,\\
\eqref{550-0}_2+\eqref{550-3}_{1,j,i}-\eqref{550-3}_{6,i,j}&=0
\end{align*}
by a direct computation, we find that the right hand side of \eqref{913} is equal to zero.
\end{proof}
\section{Two homomorphisms from the affine Yangian associated with $\widehat{\mathfrak{sl}}(n)$ to the affine Yangian associated with $\widehat{\mathfrak{sl}}(n+1)$}
In \cite{U8}, the author constructed a homomorphism from $Y_{\hbar,\ve}(\widehat{\mathfrak{sl}}(n))$ to $\widetilde{Y}_{\hbar,\ve}(\widehat{\mathfrak{sl}}(n+1))$, which is different from the one in Theorem~\ref{Main}.
\begin{Theorem}[Theorem 3.1 in \cite{U8}]\label{Pre}
There exists a homomorphism
\begin{equation*}
\widetilde{\Psi}\colon Y_{\hbar,\ve}(\widehat{\mathfrak{sl}}(n))\to \widetilde{Y}_{\hbar,\ve}(\widehat{\mathfrak{sl}}(n+1))
\end{equation*}
defined by
\begin{gather*}
\widetilde{\Psi}(X^+_{i,0})=\begin{cases}
E_{n,1}t&\text{ if }i=0,\\
E_{i,i+1}&\text{ if }i\neq 0,
\end{cases}\ 
\widetilde{\Psi}(X^-_{i,0})=\begin{cases}
E_{1,n}t^{-1}&\text{ if }i=0,\\
E_{i+1,i}&\text{ if }i\neq 0,
\end{cases}
\end{gather*}
and
\begin{align*}
\widetilde{\Psi}(H_{i,1})&= H_{i,1}-\hbar\displaystyle\sum_{s \geq 0} \limits E_{i,n+1}t^{-s-1} E_{n+1,i}t^{s+1}+\hbar\displaystyle\sum_{s \geq 0}\limits  E_{i+1,n+1}t^{-s-1} E_{n+1,i+1}t^{s+1}
\end{align*}
for $i\neq0$.
\end{Theorem}
Let us take integers $m,n\geq3$.
Combining the homomorphisms $\widetilde{\Psi}$ in Theorem~\ref{Pre} for $n,\cdots,m+n-1$, we obtain a homomorphism
\begin{gather*}
\Psi_1\colon Y_{\hbar,\ve}(\widehat{\mathfrak{sl}}(n))\to \widetilde{Y}_{\hbar,\ve}(\widehat{\mathfrak{sl}}(m+n))
\end{gather*}
given by
\begin{gather*}
\Psi_1(X^+_{i,0})=\begin{cases}
E_{n,1}t&\text{ if }i=0,\\
E_{i,i+1}&\text{ if }i\neq 0,
\end{cases}\ 
\Psi_1(X^-_{i,0})=\begin{cases}
E_{1,n}t^{-1}&\text{ if }i=0,\\
E_{i+1,i}&\text{ if }i\neq 0,
\end{cases}
\end{gather*}
and
\begin{align*}
\Psi_1(H_{i,1})&= H_{i,1}-\hbar\displaystyle\sum_{s \geq 0} \limits\sum_{k=n+1}^{m+n}\limits E_{i,k}t^{-s-1} E_{k,i}t^{s+1}+\hbar\displaystyle\sum_{s \geq 0}\limits \sum_{k=n+1}^{m+n}\limits E_{i+1,k}t^{-s-1} E_{k,i+1}t^{s+1}
\end{align*}
for $i\neq0$. Similarly, by combining the homomorphisms given in Theorem~\ref{Main} for $m,\cdots,m+n-1$, we obtain a homomorphism
\begin{gather*}
\Psi_2\colon Y_{\hbar,\ve+n\hbar}(\widehat{\mathfrak{sl}}(m))\to \widetilde{Y}_{\hbar,\ve}(\widehat{\mathfrak{sl}}(m+n))
\end{gather*}
determined by
\begin{gather*}
\Psi_2(X^+_{i,0})=\begin{cases}
E_{m+n,n+1}t&\text{ if }i=0,\\
E_{n+i,n+i+1}&\text{ if }i\neq 0,
\end{cases}\ 
\Psi_2(X^-_{i,0})=\begin{cases}
E_{n+1,m+n}t^{-1}&\text{ if }i=0,\\
E_{n+i+1,n+i}&\text{ if }i\neq 0,
\end{cases}
\end{gather*}
and
\begin{align*}
\Psi_2(H_{i,1})&= H_{i+n,1}+\hbar\displaystyle\sum_{s \geq 0}\limits\sum_{k=1}^n E_{k,n+i}t^{-s-1}E_{n+i,k}t^{s+1} -\hbar\displaystyle\sum_{s \geq 0}\limits\sum_{k=1}^n E_{k,n+i+1}t^{-s-1} E_{n+i+1,k}t^{s+1}
\end{align*}
for $i\neq0$.
\begin{Theorem}\label{Main-theorem}
The homomomorphisms $\Psi_1$ and $\Psi_2$ induce a homomorphism
\begin{equation*}
Y_{\hbar,\ve}(\widehat{\mathfrak{sl}}(n))\otimes Y_{\hbar,\ve+n\hbar}(\widehat{\mathfrak{sl}}(m))\to \widetilde{Y}_{\hbar,\ve}(\widehat{\mathfrak{sl}}(m+n)).
\end{equation*}
\end{Theorem}
\begin{proof}
Let us set
\begin{align*}
P_i&=\hbar\displaystyle\sum_{s \geq 0} \limits\sum_{k=n+1}^{m+n}\limits E_{i,k}t^{-s-1} E_{k,i}t^{s+1},\ 
Q_i=\hbar\displaystyle\sum_{s \geq 0}\limits\sum_{k=1}^n E_{k,n+i}t^{-s}E_{n+i,k}t^{s}
\end{align*}
We fix integers $1\leq i\leq n$ and $1\leq j\leq m-1$. The affine Yangian $Y_{\hbar,\ve}(\widehat{\mathfrak{sl}}(n))$ (resp. $Y_{\hbar,\ve+n\hbar}(\widehat{\mathfrak{sl}}(m))$) can be generated by $\widetilde{H}_{i,1}$ (resp. $\widetilde{H}_{j,1}$) and $\{X^\pm_{k,0}\mid 0\leq k\leq n-1\}$ (resp. $\{X^\pm_{l,0}\mid 0\leq l\leq m-1\}$). Thus, it is enough to show the commutativity between $\Psi_1(\widetilde{H}_{i,1}),\Psi_1(X^\pm_{k,0})$ and $\Psi_2(\widetilde{H}_{j,1}),\Psi_2(X^\pm_{l,0})$. The commutativity between $\Psi_1(X^\pm_{k,0})$ and $\Psi_2(X^\pm_{l,0})$ is obvious. Thus, we will show the other cases in the following three subsections.

\subsection{Commutativity between $\Psi_1(\widetilde{H}_{i,1})$ and $\Psi_2(X^\pm_{l,0})$}
We only show the $+$ case. The $-$ case can be proven in a similar way. The case that $l\neq 0$ comes from the definition of $\Psi_1(\widetilde{H}_{i,1})$. We will consider the case that $l=0$.
By a direct computation, we obtain
\begin{align}
[\Psi_1(\widetilde{H}_{i,1}),E_{m+n,n+1}t]&=[J(h_i),E_{m+n,n+1}t]-[A_i-A_{i+1},E_{m+n,n+1}t]-[P_i-P_{i+1},E_{m+n,n+1}t].\label{9115}
\end{align}
Since $E_{m+n,n+1}t=[\cdots[X^+_{0,0},X^+_{1,0}],X^+_{2,0}],\cdots,X^+_{n,0}]=0$ holds, we find that $[J(h_i),E_{m+n,n+1}t]=0$ by Lemma~\ref{J}. By the definition of $A_i$, we obtain
\begin{align}
[A_i,E_{m+n,n+1}t]&=[\dfrac{\hbar}{2}\sum_{\substack{s\geq0\\u>i}}\limits E_{u,i}t^{-s}E_{i,u}t^s,E_{m+n,n+1}t]-[\dfrac{\hbar}{2}\sum_{\substack{s\geq0\\i<v}}\limits E_{i,v}t^{-s-1}E_{v,i}t^{s+1},E_{m+n,n+1}t]\nonumber\\
&=-\dfrac{\hbar}{2}E_{m+n,i}tE_{i,n+1}-\dfrac{\hbar}{2}E_{i,n+1}E_{m+n,i}t.\label{9115-1}
\end{align}
By the definition of $P_i$, we have
\begin{align}
[P_i,E_{m+n,n+1}t]&=[\hbar\displaystyle\sum_{s \geq 0} \limits\sum_{k=n+1}^{m+n}\limits E_{i,k}t^{-s-1} E_{k,i}t^{s+1},E_{m+n,n+1}t]=\hbar E_{i,n+1} E_{m+n,i}t.\label{9115-2}
\end{align}
By applying \eqref{9115-1} and \eqref{9115-2} to \eqref{9115}, we obtain
\begin{align*}
[\Psi(\widetilde{H}_{i,1}),E_{m+n,n+1}t]&=0-\dfrac{\hbar}{2}[E_{i,n+1}, E_{m+n,i}t]+\dfrac{\hbar}{2}[E_{i+1,n+1}, E_{m+n,i+1}t]=0.
\end{align*}
\subsection{Commutativity between $\Psi_2(\widetilde{H}_{j,1})$ and $\Psi_1(X^+_{k,0})$}
We only show the $+$ case. The $-$ case can be proven in a similar way. The case that $k\neq 0$ comes from the definition of $\Psi_2(\widetilde{H}_{j,1})$. We will consider the case that $k=0$.

By a direct computation, we obtain
\begin{align}
[\Psi_2(\widetilde{H}_{j,1}),E_{n,1}t]&=[J(h_{j+n}),E_{n,1}t]-[A_{j+n}-A_{j+n+1},E_{n,1}t]+[Q_i-Q_{i+1},E_{n,1}t].\label{9116}
\end{align}
Since $E_{n,1}t=[X^+_{n,0},[\cdots,[X^+_{n-1,0},X^+_{0,0}]\cdots]$ holds, we obtain $[J(h_{j+n}),E_{n,1}t]=0$ by Lemma~\ref{J}. By the definition of $A_i$, we have
\begin{align}
[A_{j+n},E_{n,1}t]&=-[\dfrac{\hbar}{2}\sum_{\substack{s\geq0\\j+n>v}}\limits E_{j+n,v}t^{-s}E_{v,j+n}t^s,E_{n,1}t]+[\dfrac{\hbar}{2}\sum_{\substack{s\geq0\\u<j+n}}\limits E_{u,j+n}t^{-s-1}E_{j+n,u}t^{s+1},E_{n,1}t]\nonumber\\
&=-\dfrac{\hbar}{2}E_{j+n,1}tE_{n,j+n}-\dfrac{\hbar}{2} E_{n,j+n}E_{j+n,1}t.\label{9116-1}
\end{align}
By the definition of $Q_i$, we obtain
\begin{align}
[Q_j,E_{n,1}t]&=[\hbar\displaystyle\sum_{s \geq 0}\limits\sum_{k=1}^n E_{k,j+n}t^{-s-1}E_{j+n,k}t^{s+1},E_{n,1}t]=-\hbar E_{n,j+n}E_{j+n,1}t.\label{9116-2}
\end{align}
Applying \eqref{9116-1} and \eqref{9116-2} to \eqref{9116}, we obtain
\begin{align*}
[\Psi_2(\widetilde{H}_{i,1}),E_{n,1}t]&=0+\dfrac{\hbar}{2}[E_{n,j+n},E_{j+n,1}t]-\dfrac{\hbar}{2}[E_{n,j+n+1},E_{j+n+1,1}t]=0.
\end{align*}
\subsection{Commutativity between $\Psi_1(\widetilde{H}_{i,1})$ and $\Psi_2(\widetilde{H}_{j,1})$}
By \eqref{Eq2.1} and the definition of $J(h_i)$, $\Psi_1$ and $\Psi_2$, we have
\begin{align*}
&\quad[\Psi_1(\widetilde{H}_{i,1}),\Psi_2(\widetilde{H}_{j,1})]\\
&=[\widetilde{H}_{i,1},\widetilde{H}_{j+n,1}]-[P_i,\widetilde{H}_{j+n,1}]+[\widetilde{H}_{i,1},Q_j]-[P_i-P_{i+1},Q_j-Q_{j+1}]\\
&=0-[P_i-P_{i+1},J(h_{j+n})]+[P_i-P_{i+1},A_{j+n}-A_{j+n+1}]\\
&\quad+[J(h_i),Q_j-Q_{i+1}]-[A_i-A_{i+1},Q_j-Q_{j+1}]-[P_i-P_{i+1},Q_j-Q_{j+1}].
\end{align*}
By Lemma~\ref{J}, we obtain $-[P_i-P_{i+1},J(h_{j+n})]+[J(h_i),Q_j-Q_{j+1}]=0$. Thus, it is enough to show the relation
\begin{equation}
[P_i,Q_j]+[A_i,Q_j]+[A_{j+n},P_i]=0.\label{9}
\end{equation}
We will compute each terms of the left hand side of \eqref{9}. By a direct computation, we obtain
\begin{align}
&\quad[P_i,Q_j]\nonumber\\
&=\hbar^2\displaystyle\sum_{s,u \geq 0} \limits\sum_{k=1}^{m+n}\limits E_{i,k}t^{-s-1} E_{k,j+n}t^{-u-1}E_{j+n,i}t^{s+u+2}\nonumber\\
&\quad-\hbar^2\displaystyle\sum_{s,u \geq 0} \limits\sum_{k=1}^{m+n}\limits E_{i,j+n}t^{-u-s-2}E_{j+n,k}t^{u+1}E_{k,i}t^{s+1}\nonumber\\
&\quad+\hbar^2\displaystyle\sum_{s,u \geq 0} \limits\sum_{k=n+1}^{m+n}\limits E_{i,k}t^{-s-u-1} E_{k,j+n}t^{s}E_{j+n,i}t^{u+1}-\hbar^2\displaystyle\sum_{s,u \geq 0} \limits\sum_{l=1}^nE_{i,j+n}t^{-s-1} E_{l,i}t^{-u}E_{j+n,l}t^{s+u+1}\nonumber\\
&\quad+\hbar^2\displaystyle\sum_{s,u \geq 0} \limits\sum_{l=1}^nE_{l,j+n}t^{-u-s-1}E_{i,l}t^{u}E_{j+n,i}t^{s+1}\nonumber\\
&\quad-\hbar^2\displaystyle\sum_{s,u \geq 0} \limits\sum_{k=n+1}^{m+n}\limits E_{i,j+n}t^{-u-1}E_{j+n,k}t^{-s}E_{k,i}t^{u+s+1}.\label{551-0}
\end{align}
By the definition of $A_i$, we can divide $[A_i,P_j]$ into four pieces:
\begin{align}
[A_{j+n},P_i]
&=[\dfrac{\hbar}{2}\sum_{\substack{s\geq0\\u>j+n}}\limits E_{u,j+n}t^{-s}E_{j+n,u}t^s,P_i]-[\dfrac{\hbar}{2}\sum_{\substack{s\geq0\\j+n>u}}\limits E_{j+n,u}t^{-s}E_{u,j+n}t^s,P_i]\nonumber\\
&\quad+[\dfrac{\hbar}{2}\sum_{\substack{s\geq0\\u<j+n}}\limits E_{u,j+n}t^{-s-1}E_{j+n,u}t^{s+1},P_i]-[\dfrac{\hbar}{2}\sum_{\substack{s\geq0\\j+n<u}}\limits E_{j+n,u}t^{-s-1}E_{u,j+n}t^{s+1},P_i].\label{9117}
\end{align}
We compute the right hand  side of \eqref{9117}. By \eqref{AP-1} and \eqref{AP-2}, we obtain
\begin{align}
&\quad[\dfrac{\hbar}{2}\sum_{\substack{s\geq0\\u>j+n}}\limits E_{u,j+n}t^{-s}E_{j+n,u}t^s,P_i]\nonumber\\
&=-\dfrac{\hbar^2}{2}\sum_{\substack{s,v\geq0\\u>j+n}}\limits E_{u,j+n}t^{-s-v-1}E_{i,u}t^{s}E_{j+n,i}t^{v+1}+\dfrac{\hbar^2}{2}\sum_{\substack{s,v\geq0\\u>j+n}}\limits E_{i,j+n}t^{-v-1} E_{u,i}t^{-s}E_{j+n,u}t^{s+v+1},\label{1551-1}\\
&\quad-[\dfrac{\hbar}{2}\sum_{\substack{s\geq0\\j+n>u}}\limits E_{j+n,u}t^{-s}E_{u,j+n}t^s,P_i]\nonumber\\
&=-\dfrac{\hbar^2}{2}\sum_{\substack{s,v\geq0}}\limits\sum_{u=1}^n
E_{i,j+n}t^{-v-1} E_{j+n,u}t^{-s}E_{u,i}t^{s+v+1}\nonumber\\
&\quad-\dfrac{\hbar^2}{2}\sum_{\substack{s,v\geq0}}\limits\sum_{k=j+n}^{m+n}E_{j+n,k}t^{-s-v-1}E_{i,j+n}t^sE_{k,i}t^{v+1}\nonumber\\
&\quad+\dfrac{\hbar^2}{2}\sum_{\substack{s,v\geq0}}\limits\sum_{u=1}^n 
E_{i,u}t^{-s-v-1}E_{u,j+n}t^sE_{j+n,i}t^{v+1}\nonumber\\
&\quad+\dfrac{\hbar^2}{2}\sum_{\substack{s,v\geq0}}\limits\sum_{k=j+n}^{m+n}E_{i,k}t^{-v-1} E_{j+n,i}t^{-s}E_{k,j+n}t^{s+v+1},\label{1551-2}\\
&\quad[\dfrac{\hbar}{2}\sum_{\substack{s\geq0\\u<j+n}}\limits E_{u,j+n}t^{-s-1}E_{j+n,u}t^{s+1},P_i]\nonumber\\
&=\dfrac{\hbar^2}{2}\sum_{\substack{s,v\geq0}}\limits\sum_{k=1}^{m+n} E_{i,j+n}t^{-s-2-v}E_{j+n,k}t^{s+1}E_{k,i}t^{v+1}\nonumber\\
&\quad-\dfrac{\hbar^2}{2}\sum_{\substack{s,v\geq0}}\limits\sum_{k=1}^{m+n} E_{i,k}t^{-v-1}E_{k,j+n}t^{-s-1}E_{j+n,i}t^{s+v+2}\nonumber\\
&\quad+\dfrac{\hbar^2}{2}\sum_{\substack{s,v\geq0}}\limits\sum_{k=n+1}^{m+n} E_{i,j+n}t^{-s-1}E_{j+n,k}t^{-v}E_{k,i}t^{s+v+1}-\dfrac{\hbar^2}{2}\sum_{\substack{s,v\geq0\\u<j+n}}\limits E_{u,j+n}t^{-s-v-1}E_{i,u}t^{s}E_{j+n,i}t^{v+1}\nonumber\\
&\quad+\dfrac{\hbar^2}{2}\sum_{\substack{s,v\geq0\\u<j+n}}\limits E_{i,j+n}t^{-v-1} E_{u,i}t^{-s}E_{j+n,u}t^{s+v+1}-\dfrac{\hbar^2}{2}\sum_{\substack{s,v\geq0}}\limits\sum_{k=n+1}^{m+n} E_{i,k}t^{-s-v-1}E_{k,j+n}t^{v}E_{j+n,i}t^{s+1}\label{1551-3}\\
&\quad-[\dfrac{\hbar}{2}\sum_{\substack{s,v\geq0\\j+n<u}}\limits E_{j+n,u}t^{-s-1}E_{u,j+n}t^{s+1},P_i]\nonumber\\
&=\dfrac{\hbar^2}{2}\sum_{\substack{s,v\geq0\\j+n<u}}\limits E_{j+n,u}t^{-s-v-1}E_{i,j+n}t^{s}E_{u,i}t^{v+1}-\dfrac{\hbar^2}{2}\sum_{\substack{s,v\geq0\\j+n<u}}\limits E_{i,u}t^{-v-1}E_{j+n,i}t^{-s}E_{u,j+n}t^{s+v+1}.\label{1551-4}
\end{align}
By the definition of $Q_i$, we have
\begin{align}
[A_i,Q_j]
&=[\dfrac{\hbar}{2}\sum_{\substack{s\geq0\\u>i}}\limits E_{u,i}t^{-s}E_{i,u}t^s,Q_j]-[\dfrac{\hbar}{2}\sum_{\substack{s\geq0\\i>u}}\limits E_{i,u}t^{-s}E_{u,i}t^s,Q_j]\nonumber\\
&\quad+[\dfrac{\hbar}{2}\sum_{\substack{s\geq0\\u<i}}\limits E_{u,i}t^{-s-1}E_{i,u}t^{s+1},Q_j]-[\dfrac{\hbar}{2}\sum_{\substack{s\geq0\\i<u}}\limits E_{i,u}t^{-s-1}E_{u,i}t^{s+1},Q_j].\label{9112}
\end{align}
We compute the right hand side of \eqref{9112}. By a direct computation, we obtain
\begin{align}
&\quad[\dfrac{\hbar}{2}\sum_{\substack{s\geq0\\u>i}}\limits E_{u,i}t^{-s}E_{i,u}t^s,Q_j]\nonumber\\
&=-\dfrac{\hbar^2}{2}\sum_{\substack{s,v\geq0\\}}\limits\sum_{l=1}^{i} E_{l,i}t^{-s-v-1}E_{i,j+n}t^sE_{j+n,l}t^{v+1}\nonumber\\
&\quad-\dfrac{\hbar^2}{2}\sum_{\substack{s,v\geq0}}\limits\sum_{u=n+1}^{m+n} E_{i,j+n}t^{-v-1}E_{u,i}t^{-s}E_{j+n,u}t^{s+v+1}\nonumber\\
&\quad+\dfrac{\hbar^2}{2}\sum_{\substack{s,v\geq0}}\limits\sum_{u=n+1}^{m+n} E_{u,j+n}t^{-s-v-1}E_{i,u}t^sE_{j+n,i}t^{v+1}\nonumber\\
&\quad+\dfrac{\hbar^2}{2}\sum_{\substack{s,v\geq0}}\limits\sum_{l=1}^i E_{l,j+n}t^{-v-1}E_{j+n,i}t^{-s}E_{i,l}t^{s+v+1},\label{551-1}\\
&\quad-[\dfrac{\hbar}{2}\sum_{\substack{s\geq0\\i>u}}\limits E_{i,u}t^{-s}E_{u,i}t^s,Q_j]\nonumber\\
&=-\dfrac{\hbar^2}{2}\sum_{\substack{s,v\geq0\\i>u}}\limits E_{i,u}t^{-s-v-1}E_{u,j+n}t^{s}E_{j+n,i}t^{v+1}+\dfrac{\hbar^2}{2}\sum_{\substack{s,v\geq0\\i>u}}\limits E_{i,j+n}t^{-v-1}E_{j+n,u}t^{-s}E_{u,i}t^{s+v+1},\label{551-2}\\
&\quad[\dfrac{\hbar}{2}\sum_{\substack{s\geq0\\u<i}}\limits E_{u,i}t^{-s-1}E_{i,u}t^{s+1},Q_j]\nonumber\\
&=\dfrac{\hbar^2}{2}\sum_{\substack{s,v\geq0\\u<i}}\limits E_{u,i}t^{-s-1-v}E_{i,j+n}t^{s}E_{j+n,u}t^{v+1}-\dfrac{\hbar^2}{2}\sum_{\substack{s,v\geq0\\u<i}}\limits E_{u,j+n}t^{-v-1}E_{j+n,i}t^{-s}E_{i,u}t^{v+s+1},\label{551-3}\\
&\quad[\dfrac{\hbar}{2}\sum_{\substack{s\geq0\\u<i}}\limits E_{u,i}t^{-s-1}E_{i,u}t^{s+1},Q_j]\nonumber\\
&=-\dfrac{\hbar^2}{2}\sum_{\substack{s,v\geq0}}\limits\sum_{l=1}^{m+n}\limits E_{l,j+n}t^{-v-1}E_{i,l}t^{-s-1}E_{j+n,i}t^{v+s+2}\nonumber\\
&\quad+\dfrac{\hbar^2}{2}\sum_{\substack{s,v\geq0}}\limits\sum_{l=1}^{m+n} E_{i,j+n}t^{-s-v-2}E_{l,i}t^{s+1}E_{j+n,l}t^{v+1}\nonumber\\
&\quad-\dfrac{\hbar^2}{2}\sum_{\substack{s,v\geq0\\i<u}}\limits E_{i,u}t^{-s-v-1}E_{u,j+n}t^{s}E_{j+n,i}t^{v+1}+\dfrac{\hbar^2}{2}\sum_{\substack{s,v\geq0}}\limits\sum_{l=1}^n E_{i,j+n}t^{-s-1}E_{l,i}t^{-v}E_{j+n,l}t^{s+v+1}\nonumber\\
&\quad-\dfrac{\hbar^2}{2}\sum_{\substack{s,v\geq0}}\limits\sum_{l=1}^n E_{l,j+n}t^{-v-s-1}E_{i,l}t^{v}E_{j+n,i}t^{s+1}+\dfrac{\hbar^2}{2}\sum_{\substack{s,v\geq0\\i<u}}\limits E_{i,j+n}t^{-v-1}E_{j+n,u}t^{-s}E_{u,i}t^{v+s+1}.\label{551-4}
\end{align}
We compute the sum of \eqref{551-0}, \eqref{1551-1}-\eqref{1551-4} and \eqref{551-1}-\eqref{551-4} into eight picies as follows: 
\begin{align*}
&\quad\eqref{1551-2}_2+\eqref{1551-4}_1+\eqref{551-1}_1+\eqref{551-3}_1\\
&=-\dfrac{\hbar^2}{2}\sum_{\substack{s,v\geq0}}\limits E_{j+n,j+n}t^{-s-v-1}E_{i,j+n}t^sE_{j+n,i}t^{v+1}-\dfrac{\hbar^2}{2}\sum_{\substack{s,v\geq0\\}}\limits E_{i,i}t^{-s-v-1}E_{i,j+n}t^sE_{j+n,i}t^{v+1},\\
&\quad\eqref{551-0}_2+\eqref{1551-3}_1+\eqref{551-4}_2=-(m+n)\dfrac{\hbar^2}{2}\sum_{\substack{s\geq0}}\limits (s+1)E_{i,j+n}t^{-s-2}E_{j+n,i}t^{s+2},\\
&\quad\eqref{551-0}_4+\eqref{1551-1}_2+\eqref{1551-3}_5+\eqref{551-1}_2+\eqref{551-4}_4\\
&=-\dfrac{\hbar^2}{2}\sum_{\substack{s,v\geq0}}\limits E_{i,j+n}t^{-v-1}E_{j+n,i}t^{-s}E_{j+n,j+n}t^{s+v+1},\\
&\quad\eqref{551-0}_6+\eqref{1551-2}_1+\eqref{1551-3}_3+\eqref{551-2}_2+\eqref{551-4}_6=-\dfrac{\hbar^2}{2}\sum_{\substack{s,v\geq0}}\limits E_{i,j+n}t^{-v-1} E_{j+n,i}t^{-s}E_{i,i}t^{s+v+1},\\
&\quad\eqref{551-0}_1+\eqref{1551-3}_2+\eqref{551-4}_1=(m+n)\dfrac{\hbar^2}{2}\sum_{\substack{s\geq0}}\limits (s+1)E_{i,j+n}t^{-s-2}E_{j+n,i}t^{s+2},\\
&\quad\eqref{551-0}_3+\eqref{1551-2}_3+\eqref{1551-3}_6+\eqref{551-2}_1+\eqref{551-4}_3=\dfrac{\hbar^2}{2}\displaystyle\sum_{s,u \geq 0} \limits E_{i,i}t^{-s-u-1} E_{i,j+n}t^{s}E_{j+n,i}t^{u+1},\\
&\quad\eqref{551-0}_5+\eqref{1551-1}_1+\eqref{1551-3}_4+\eqref{551-1}_3+\eqref{551-4}_5=\dfrac{\hbar^2}{2}\displaystyle\sum_{s,u \geq 0} \limits E_{j+n,j+n}t^{-u-s-1}E_{i,j+n}t^{u}E_{j+n,i}t^{s+1},\\
&\quad\eqref{1551-2}_4+\eqref{1551-4}_2+\eqref{551-1}_4+\eqref{551-3}_1\\
&=\dfrac{\hbar^2}{2}\sum_{\substack{s\geq0}}\limits E_{i,j+n}t^{-v-1} E_{j+n,i}t^{-s}E_{j+n,j+n}t^{s+v+1}+\dfrac{\hbar^2}{2}\sum_{\substack{s,v\geq0}}\limits E_{i,j+n}t^{-v-1}E_{j+n,i}t^{-s}E_{i,i}t^{s+v+1}.
\end{align*}
Since the sum of the eight equations above is equal to zero, we have shown \eqref{9}.
\end{proof}

\section{Application to the evaluation map for the affine Yangian}
The evaluation map for the affine Yangian is a a non-trivial homomorphism from the affine Yangian $Y_{\hbar,\ve}(\widehat{\mathfrak{sl}}(n))$ to the completion of the universal enveloping algebra of the affinization of $\mathfrak{gl}(n)$. We set a Lie algebra 
\begin{equation*}
\widehat{\mathfrak{gl}}(n)=\mathfrak{gl}(n)\otimes\mathbb{C}[z^{\pm1}]\oplus\mathbb{C}\tilde{c}\oplus\mathbb{C}z
\end{equation*}
whose commutator relations are given by
\begin{gather*}
[x\otimes t^u, y\otimes t^v]=\begin{cases}
[x,y]\otimes t^{u+v}+\delta_{u+v,0}u\text{tr}(xy)\tilde{c}\ \text{ if }x,y\in\mathfrak{sl}(n),\\
[e_{a,b},e_{i,i}]\otimes t^{u+v}+\delta_{u+v,0}u\text{tr}(E_{a,b}E_{i,i})\tilde{c}+\delta_{u+v,0}\delta_{a,b}uz\\
\qquad\qquad\qquad\qquad\qquad\qquad\qquad\qquad\qquad\text{ if }x=e_{a,b},\ y=e_{i,i},
\end{cases}\\
\text{$z$ and $\tilde{c}$ are central elements of }\widehat{\mathfrak{gl}}(n),
\end{gather*}
where tr is a trace of $\mathfrak{gl}(n)$, that is, $\text{tr}(E_{i,j}E_{k,l})=\delta_{i,l}\delta_{j,k}$. 

We consider a completion of $U(\widehat{\mathfrak{gl}}(n))/U(\widehat{\mathfrak{gl}}(n))(z-1)$ following \cite{MNT} and \cite{GNW}. 
We take the grading of $U(\widehat{\mathfrak{gl}}(n))/U(\widehat{\mathfrak{gl}}(n))(z-1)$ as $\text{deg}(X(s))=s$ and $\text{deg}(\tilde{c})=0$. We denote the degreewise completion of $U(\widehat{\mathfrak{gl}}(n))/U(\widehat{\mathfrak{gl}}(n))(z-1)$ by $\mathcal{U}(\widehat{\mathfrak{gl}}(n))$.
\begin{Theorem}[Theorem 3.8 in \cite{K1} and Theorem 4.18 in \cite{K2}]\label{thm:main}
\begin{enumerate}
\item Suppose that $\tilde{c} =\dfrac{\ve}{\hbar}$.
Then, there exists an algebra homomorphism 
\begin{equation*}
\ev_{\hbar,\ve}^n \colon Y_{\hbar,\ve}(\widehat{\mathfrak{sl}}(n)) \to \mathcal{U}(\widehat{\mathfrak{gl}}(n))
\end{equation*}
uniquely determined by 
\begin{gather*}
	\ev_{\hbar,\ve}^n(X_{i,0}^{+}) = \begin{cases}
E_{n,1}t&\text{ if }i=0,\\
E_{i,i+1}&\text{ if }1\leq i\leq n-1,
\end{cases} \ev_{\hbar,\ve}^n(X_{i,0}^{-}) = \begin{cases}
E_{1,n}t^{-1}&\text{ if }i=0,\\
E_{i+1,i}&\text{ if }1\leq i\leq n-1,
\end{cases}\\ \ev_{\hbar,\ve}^n(H_{i,0}) =\begin{cases}
E_{n,n}-E_{1,1}+\tilde{c}&\text{ if }i=0,\\
E_{i,i}-E_{i+1,i+1}&\text{ if }1\leq i\leq n-1.
\end{cases}
\end{gather*}
and
\begin{align*}
\ev_{\hbar,\ve}^n(H_{i,1}) &=-\dfrac{i}{2}\hbar \ev_{\hbar,\ve}^n(H_{i,0}) -\hbar E_{i,i}E_{i+1,i+1} \\
&\quad+ \hbar \displaystyle\sum_{s \geq 0}  \limits\displaystyle\sum_{k=1}^{i}\limits  E_{i,k}t^{-s}E_{k,i}t^s+\hbar \displaystyle\sum_{s \geq 0} \limits\displaystyle\sum_{k=i+1}^{n}\limits  E_{i,k}t^{-s-1}E_{k,i}t^{s+1}\\
&\quad-\hbar\displaystyle\sum_{s \geq 0}\limits\displaystyle\sum_{k=1}^{i}\limits E_{i+1,k}t^{-s} E_{k,i+1}t^{s}-\hbar\displaystyle\sum_{s \geq 0}\limits\displaystyle\sum_{k=i+1}^{n} \limits E_{i+1,k}t^{-s-1} E_{k,i+1}t^{s+1}
\end{align*}
for $i\neq0$.
\item In the case that $\ve\neq 0$, the image of the evaluation map is dense in $\mathcal{U}(\widehat{\mathfrak{gl}}(n))$.
\end{enumerate}
\end{Theorem}
We note that $U(\widehat{\mathfrak{gl}}(n))$ can be embedded into $U(\widehat{\mathfrak{gl}}(m+n))$ by $\tilde{c}\mapsto\tilde{c}$ and $E_{i,j}t^s\mapsto E_{i,j}t^s$ for $i\neq j$. By the definition of the evaluation map and $\Psi_1$, we obtain the following theorem.
\begin{Theorem}
The following relation holds:
\begin{equation*}
\ev^{m+n}_{\hbar,\ve}\circ\Psi_1=\ev^n_{\hbar,\ve}.
\end{equation*}
\end{Theorem}
Let us set the centralizer algebra
\begin{align*}
C(\widehat{\mathfrak{gl}}(m+n),\widehat{\mathfrak{gl}}(n))&=\{x\in \mathcal{U}(\widehat{\mathfrak{gl}}(m+n))\mid[x,U(\widehat{\mathfrak{gl}}(n))]=0\}.
\end{align*}
Theorem~\ref{Main-theorem} and Theorem~\ref{thm:main} induces the following corollary.
\begin{Corollary}
In the case that $\ve\neq0$, we obtain a homomorphism
\begin{equation*}
\ev^{m+n}_{\hbar,\ve}\circ\Psi_2\colon Y_{\hbar,\ve+n\hbar}(\widehat{\mathfrak{sl}}(m))\to C(\widehat{\mathfrak{gl}}(m+n),\widehat{\mathfrak{gl}}(n)).
\end{equation*}
\end{Corollary}
This result can be interpreted from the perspective of a vertex algebra. For a vertex algebra $V$, we denote the generating field associated with $v\in V$ by $v(z)=\displaystyle\sum_{n\in\mathbb{Z}}\limits v_{(n)}z^{-n-1}$. We also denote the OPE of $V$ by
\begin{equation*}
u(z)v(w)\sim\displaystyle\sum_{s\geq0}\limits \dfrac{(u_{(s)}v)(w)}{(z-w)^{s+1}}
\end{equation*}
for all $u, v\in V$. We denote the vacuum vector (resp.\ the translation operator) by $|0\rangle$ (resp.\ $\partial$).

Let us recall the definition of a universal enveloping algebra of a vertex algebra in the sense of \cite{FZ} and \cite{MNT}.
For any vertex algebra $V$, let $L(V)$ be the Borcherds Lie algebra, that is,
\begin{align}
 L(V)=V{\otimes}\mathbb{C}[t,t^{-1}]/\text{Im}(\partial\otimes\id +\id\otimes\frac{d}{d t})\label{844},
\end{align}
where the commutation relation is given by
\begin{align*}
 [ut^a,vt^b]=\sum_{r\geq 0}\begin{pmatrix} a\\r\end{pmatrix}(u_{(r)}v)t^{a+b-r}
\end{align*}
for all $u,v\in V$ and $a,b\in \mathbb{Z}$. 
\begin{Definition}[Section~6 in \cite{MNT}]\label{Defi}
We set $\mathcal{U}(V)$ as the quotient algebra of the standard degreewise completion of the universal enveloping algebra of $L(V)$ by the completion of the two-sided ideal generated by
\begin{gather}
(u_{(a)}v)t^b-\sum_{i\geq 0}
\begin{pmatrix}
 a\\i
\end{pmatrix}
(-1)^i(ut^{a-i}vt^{b+i}-(-1)^avt^{a+b-i}ut^{i}),\label{241}\\
|0\rangle t^{-1}-1.\label{242}
\end{gather}
We call $\mathcal{U}(V)$ the universal enveloping algebra of $V$.

We denote the universal affine vertex algebra associated with a finite dimensional Lie algebra $\mathfrak{g}$ and its inner product $\kappa$ by $V^\kappa(\mathfrak{g})$. By the PBW theorem, we can identify $V^\kappa(\mathfrak{g})$ with $U(t^{-1}\mathfrak{g}[t^{-1}])$. In order to simplify the notation, here after, we denote the generating field $(ut^{-1})(z)$ as $u(z)$. By the definition of $V^\kappa(\mathfrak{g})$, the generating fields $u(z)$ and $v(z)$ satisfy the OPE
\begin{gather}
u(z)v(w)\sim\dfrac{[u,v](w)}{z-w}+\dfrac{\kappa(u,v)}{(z-w)^2}\label{OPE1}
\end{gather}
for all $u,v\in\mathfrak{g}$. 
Let us set an inner product on $\mathfrak{gl}(m)\subset\mathfrak{gl}(m+n)$ by
\begin{equation*}
\kappa(E_{i,j},E_{p,q})=\delta_{i,q}\delta_{p,j}\tilde{c}+\delta_{i,j}\delta_{p,q}.
\end{equation*}
Then, we find that $\mathcal{U}(V^\kappa(\mathfrak{gl}(m+n)))$ and  $\mathcal{U}(V^\kappa(\mathfrak{gl}(n)))$ coincide with $\mathcal{U}(\widehat{\mathfrak{gl}}(m+n)))$ and $\mathcal{U}(\widehat{\mathfrak{gl}}(n)))$. We denote $E_{i,j}t^{-s}\in U(t^{-1}\mathfrak{gl}(n)[t^{-1}])=V^\kappa(\mathfrak{gl}(n))$ by $E_{i,j}[-s]$.

For a vertex algebra $A$ and its vertex subalgebra $B$, we set a coset vertex algebra of the pair $(A,B)$ as follows:
\begin{align*}
C(A,B)=\{v\in A\mid w_{(r)}v=0\text{ for }w\in B\text{ and }r\geq0\}.
\end{align*}
\begin{Theorem}
The homomorphism $\ev^{m+n}_{\hbar,\ve}\circ\Psi_2$ induces the homomorphism
\begin{equation*}
\ev^{m+n}_{\hbar,\ve}\circ\Psi_2\colon Y_{\hbar,\ve+n\hbar}(\widehat{\mathfrak{sl}}(m))\to \mathcal{U}(C(V^\kappa(\mathfrak{gl}(m+n)),V^\kappa(\mathfrak{sl}(n)))).
\end{equation*}
\end{Theorem}
\begin{proof}
By the definition of the universal affine vertex algebra and coset, we find that $E_{i,j}[-1]$ is contained in $C(V^\kappa(\mathfrak{gl}(m+n)),V^\kappa(\mathfrak{sl}(n)))$ for $i,j\geq n+1$. Thus, the image of $\ev^{m+n}_{\hbar,\ve}\circ\Psi_2$ is contained in $\mathcal{U}(C(V^\kappa(\mathfrak{gl}(m+n)),V^\kappa(\mathfrak{sl}(n))))$ if $\sum_{u=1}^n(E_{u,i}[-1])_{(-1)}E_{j,u}[-1]$ is contained in $C(V^\kappa(\mathfrak{gl}(m+n)),V^\kappa(\mathfrak{sl}(n)))$ for $i,j\geq n+1$. By a direct compuation, we obtain
\begin{gather*}
x_{(r)}(\sum_{u=1}^n(E_{u,i}[-1])_{(-1)}E_{j,u}[-1])=0\text{ if }x\in\mathfrak{sl}(n)\text{ and }r\geq0.
\end{gather*}
Then, we find that $\sum_{u=1}^n(E_{u,i}[-1])_{(-1)}E_{j,u}[-1]$ is contained in $C(V^\kappa(\mathfrak{gl}(m+n)),V^\kappa(\mathfrak{sl}(n)))$ for $i,j\geq n+1$.
\end{proof}
\section{Application to the rectangular $W$-algebra}
The $W$-algebra $\mathcal{W}^k(\mathfrak{g},f)$ is a vertex algebra associated with a finite dimensional reductive Lie algebra $\mathfrak{g}$ and a nilpotent element $f$. We call the $W$-algebra associated with $\mathfrak{gl}(ln)$ and a nilpotent element of type $(l^n)$ the rectangular $W$-algebra and denote it by $\mathcal{W}^k(\mathfrak{gl}(ln),(l^n))$.
In this article, we only consider the case that $l=2$. The nilpotent element is
\begin{equation*}
f=\sum_{u=1}^n\limits E_{n+u,u}\in\mathfrak{gl}(2n).
\end{equation*}
We set the inner product on $\mathfrak{gl}(n)$ by
\begin{equation*}
\kappa(E_{i,j},E_{p,q})=\delta_{j,p}\delta_{i,q}\alpha+\delta_{i,j}\delta_{p,q},
\end{equation*}
where $\alpha=k+n$. 

By Theorem 3.1 and Corollary 3.2 in \cite{AM}, we obtain the following theorem.
\begin{Theorem}[Corollary 5.2 in \cite{Genra}, Theorem 3.1 and Corollary 3.2 in \cite{AM} and Section 4 in \cite{U8}]\label{AM}

\textup{(1)}\ 
The rectangular $W$-algebra $\mathcal{W}^k(\mathfrak{gl}(2n),(2^n))$ can be realized as a vertex subalgebra of $V^{\kappa}(\mathfrak{gl}(n))^{\otimes 2}$.

\textup{(2)}\ The $W$-algebra $\mathcal{W}^k(\mathfrak{g},f)$ has the following strong generators:
\begin{align*}
W^{(1)}_{i,j}&=E^{(1)}_{i,j}[-1]+E^{(2)}_{i,j}[-1],\\
W^{(2)}_{i,j}&=\sum_{1\leq u\leq n}\limits E^{(1)}_{u,j}[-1]E^{(2)}_{i,u}[-1]+\alpha E^{(2)}_{i,j}[-1]
\end{align*} 
for $1\leq i,j\leq n$, where $E^{(1)}_{i,j}[-1]=E_{i,j}[-1]\otimes 1\in V^{\kappa}(\mathfrak{gl}(n))^{\otimes 2}$ and $E^{(2)}_{i,j}[-1]=1\otimes E_{i,j}[-1]\in V^{\kappa}(\mathfrak{gl}(n))^{\otimes 2}$.

\textup{(3)}\ There exists the embedding determined by
\begin{equation*}
\iota\colon\mathcal{W}^{k+m}(\mathfrak{gl}(2n),(2^n))\to \mathcal{W}^{k}(\mathfrak{gl}(2m+2n),(2^{m+n})),\ W^{(u)}_{i,j}\mapsto W^{(u)}_{i,j}.
\end{equation*}

\textup{(4)}\ The $W$-algebra $\mathcal{W}^{k}(\mathfrak{gl}(2m+2n),(2^{m+n}))$ has a subalgebra isomorphic to $\mathcal{W}^{k}(\mathfrak{gl}(2n),(2^n))$.
\end{Theorem}
In \cite{U4} Theorem~5.1, the author constructed a surjective homomorphism from the affine (super) Yangian to the universal enveloping algebra of a rectangular $W$-(super)algebra.

\end{Definition}
\begin{Theorem}[Theorem 5.1 in \cite{U4} and Theorem 9.2 in \cite{KU}]\label{Maim}
\begin{enumerate}
\item
Suppose that $\ve=\hbar\alpha$.
There exists an algebra homomorphism 
\begin{equation*}
\Phi^{n}\colon Y_{\hbar,\ve}(\widehat{\mathfrak{sl}}(n))\to \mathcal{U}(\mathcal{W}^{k}(\mathfrak{gl}(2n),(2^{n})))
\end{equation*} 
determined by
\begin{gather*}
\Phi^n(H_{i,0})=\begin{cases}
W^{(1)}_{n,n}-W^{(1)}_{1,1}+2\alpha&\text{ if }i=0,\\
W^{(1)}_{i,i}-W^{(1)}_{i+1,i+1}&\text{ if }i\neq 0,
\end{cases}\\
\Phi^n(X^+_{i,0})=\begin{cases}
W^{(1)}_{n,1}t&\text{ if }i=0,\\
W^{(1)}_{i,i+1}&\text{ if }i\neq 0,
\end{cases}
\quad \Phi^n(X^-_{i,0})=\begin{cases}
W^{(1)}_{1,n}t^{-1}&\text{ if }i=0,\\
W^{(1)}_{i+1,i}&\text{ if }i\neq 0,
\end{cases}
\end{gather*}
and
\begin{align*}
\Phi^n(H_{i,1})&=
-\hbar W^{(2)}_{i,i}t+\hbar W^{(2)}_{i+1,i+1}t+\dfrac{i}{2}\hbar\Phi^n(H_{i,0})-\hbar W^{(1)}_{i,i}W^{(1)}_{i+1,i+1}\\
&\quad+\hbar \displaystyle\sum_{s \geq 0}  \limits\displaystyle\sum_{u=1}^{i}\limits W^{(1)}_{i,u}t^{-s}W^{(1)}_{u,i}t^s+\hbar\displaystyle\sum_{s \geq 0} \limits\displaystyle\sum_{u=i+1}^{n}\limits  W^{(1)}_{i,u}t^{-s-1} W^{(1)}_{u,i}t^{s+1}\\
&\quad-\hbar\displaystyle\sum_{s \geq 0}\limits\displaystyle\sum_{u=1}^{i}\limits W^{(1)}_{i+1,u}t^{-s} W^{(1)}_{u,i+1}t^s-\hbar\displaystyle\sum_{s \geq 0}\limits\displaystyle\sum_{u=i+1}^{n} \limits W^{(1)}_{i+1,u}t^{-s-1} W^{(1)}_{u,i+1}t^{s+1}
\end{align*}
for $i\neq 0$.
\item In the case that $\ve\neq0$, the image of $\Phi^n$ is dense in $\mathcal{U}(\mathcal{W}^{k}(\mathfrak{gl}(2n),(2^{n})))$.
\end{enumerate}
\end{Theorem}
By the definition of $\Phi^n$, we obtain the following relation.
\begin{Theorem}[Theorem 5.6 in \cite{U8}]\label{u8}
Suppose that $\ve=-(k+(n+m))\hbar$. We obtain the following relation:
\begin{equation*}
\Phi^{m+n}\circ\Psi_1=\iota\circ\Phi^n.
\end{equation*}
\end{Theorem}

We can consider the coset vertex algebra $C(\mathcal{W}^k(\mathfrak{gl}(2m+2n),(2^{m+n}),\mathcal{W}^{k+m}(\mathfrak{sl}(2n),(2^{n})))$ by Theorem~\ref{AM}. Then, Theorem~\ref{Main-theorem} and Theorem~\ref{u8} induce the following corollary.
\begin{Corollary}\label{Corr}
In the case that $\ve\neq0$, we obtain
\begin{equation*}
\Phi^{m+n}_{\hbar,\ve}\circ\Psi_2\colon Y_{\hbar,\ve+n\hbar}(\widehat{\mathfrak{sl}}(m))\to C(\mathcal{U}(\mathcal{W}^k(\mathfrak{gl}(2m+2n),(2^{m+n})),\mathcal{U}(\mathcal{W}^{k+m}(\mathfrak{gl}(2n),(2^{n})))),
\end{equation*}
where
\begin{align*}
&\quad C(\mathcal{U}(\mathcal{W}^k(\mathfrak{gl}(2m+2n),(2^{m+n}))),\mathcal{U}(\mathcal{W}^{k+m}(\mathfrak{gl}(2n),(2^{n}))))\\
&=\{x\in\mathcal{U}(\mathcal{W}^k(\mathfrak{gl}(2m+2n),(2^{m+n})))\mid[x,\mathcal{U}(\mathcal{W}^{k+m}(\mathfrak{gl}(2n),(2^{n})))]=0\}.
\end{align*}
\end{Corollary}
We can also 
\begin{Theorem}
In the case that $\ve\neq0$, the homomorphism induces the homomorphism
\begin{equation*}
\Phi^{m+n}_{\hbar,\ve}\circ\Psi_2\colon Y_{\hbar,\ve+n\hbar}(\widehat{\mathfrak{sl}}(m))\to \mathcal{U}(C(\mathcal{W}^k(\mathfrak{gl}(2m+2n),(2^{m+n})),\mathcal{W}^{k+m}(\mathfrak{sl}(2n),(2^{n})))).
\end{equation*}
\end{Theorem}
\begin{proof}
By the definition, we find that $W^{(1)}_{i,j}[-1]$ is contained in the universal enveloping algebra of the coset for $i,j\geq n+1$. The OPEs of $W^{(1)}_{i,j}$ and $\mathcal{W}^{k+m}(\mathfrak{sl}(2n),(2^{n}))$ are non-zero for $i,j\geq n+1$ due to the inner product $(E_{i,i},E_{p,p})=1$ for $i\geq n+1,p\leq n$.
Since $(E_{i,i}-E_{i+1,i+1},E_{p,p}-E_{p+1,p+1})=0$ for $i\geq n+1,p\leq n-1$, we find that
\begin{equation*}
[W^{(1)}_{i,j}t^s,\mathcal{U}(\mathcal{W}^{k+m}(\mathfrak{sl}(2n),(2^{n})))]=0
\end{equation*}
for $i,j\geq n+1$.
Thus, the image of $\Phi^{m+n}_{\hbar,\ve}\circ\Psi_2$ is contained in the coset if $W^{(2)}_{i,i}-W^{(2)}_{i+1,i+1}$ is contained in the coset for $i\geq n+1$. 
By the OPEs of the rectangular $W$-algebra $(\mathcal{W}^k(\mathfrak{gl}(2m+2n),(2^{m+n}))$ given in \cite{U4}, we have
\begin{align}
&\quad[[(W^{(2)}_{i,i}-W^{(2)}_{i+1,i+1})t,W^{(1)}_{i,i+1}t^{-2}],W^{(1)}_{i+1,i}]\nonumber\\
&=(W^{(2)}_{i,i}-W^{(2)}_{i+1,i+1})t^{-1}+\text{the elements generated by $W^{(1)}_{p,q}t^s$ for $p,q\geq n+1$ and $s\in\mathbb{Z}$}.\label{align1234}
\end{align}
By Corollary~\ref{Corr}, the left hand side of \eqref{align1234} is contained in the centralizer of $\mathcal{U}(\mathcal{W}^{k+m}(\mathfrak{gl}(2n),(2^{n}))$ with $\mathcal{U}(\mathcal{W}^k(\mathfrak{gl}(2m+2n),(2^{m+n}))$. Thus, we obtain
\begin{gather*}
[xt^r,(W^{(2)}_{i,i}-W^{(2)}_{i+1,i+1})t^{-1}]=0\text{ if }x\in\mathcal{W}^{k+m}(\mathfrak{sl}(2n),(2^{n})).
\end{gather*}
Then, we find that $W^{(2)}_{i,i}-W^{(2)}_{i+1,i+1}$ is contained in the coset. 
\end{proof}
At the last of this section, we will note the relationship between $\Psi_1,\Psi_2$ and two embeddings of the finite Yangian of type $A$. The Yangian associated with $\mathfrak{sl}(n)$ is the associative algebra whose generators are
\begin{equation*}
\{H_{i,r},X^\pm_{i,r}\mid 1\leq i\leq n-1, r=0,1\}
\end{equation*}
with the defining relations \eqref{Eq2.1}-\eqref{Eq2.5}, \eqref{Eq2.8} and \eqref{Eq2.10}. We denote the Yangian associated with $\mathfrak{sl}(n)$ by $Y_\hbar(\mathfrak{sl}(n))$. The Yangian $Y_\hbar(\mathfrak{sl}(n))$ can be naturally embedded into $Y_{\hbar,\ve}(\widehat{\mathfrak{sl}}(n))$ and we identify $Y_\hbar(\mathfrak{sl}(n))$ with the corresponding subalgebra. By the defining relations, we obtain two embeddings:
\begin{gather*}
\Psi_1^{\text{fin}}\colon Y_\hbar(\mathfrak{sl}(n))\to Y_\hbar(\mathfrak{sl}(m+n)),\ A_{i,r}\mapsto A_{i,r},\\
\Psi_2^{\text{fin}}\colon Y_\hbar(\mathfrak{sl}(m))\to Y_\hbar(\mathfrak{sl}(m+n))A_{i,r}\mapsto A_{i+m,r}
\end{gather*}
for $A=H,X^\pm$. We note that $\Psi_2^{\text{fin}}$ corresponds to the homomorphism $\psi_n$ defined (4.2) in \cite{BK0}. In Section 6 of \cite{U9}, we show that $\Psi_1$ is the affine analogue of $\Psi_1^{\text{fin}}$. Here after, we will show the similar result for $\Psi_2$. 

In \cite{BK}, Brundan-Kleshchev wrote down a finite $W$-algebra of type $A$ as a quotient algebra of the shifted Yangian, which is a subalgebra of the Yangian associated with $\mathfrak{gl}(n)$. Especially, Brundan-Kleshchev's work gave a homomorphism
\begin{gather*}
\Phi^{n,\text{fin}}\colon Y_\hbar(\mathfrak{sl}(n))\to\mathcal{W}^{\text{fin}}(\mathfrak{gl}(2n),(2^n)),
\end{gather*}
where $\mathcal{W}^{\text{fin}}(\mathfrak{gl}(2n),(2^n))$ is a finite $W$-algebra associated with $\mathfrak{gl}(2n)$ and a nilpotent element of type $(2^n)$. In Section 6 of \cite{U4}, we show the relation:
\begin{equation*}
\Phi^{n,\text{fin}}=p\circ\Phi^n_{\hbar,\ve}|_{Y_\hbar(\mathfrak{sl}(n))},
\end{equation*}
where $p$ is a natural projection from $\mathcal{U}(\mathcal{W}^{k}(\mathfrak{gl}(2n),(2^n)))$ to $\mathcal{W}^{\text{fin}}(\mathfrak{gl}(2n),(2^n))$. By the definition of $p$ and a direct computation, we can show that
\begin{equation*}
\Phi^{m+n,\text{fin}}\circ\Psi_2^{\text{fin}}(A_{i,r})=p\circ\Phi^{m+n}\circ\Psi_2(A_{i,r})
\end{equation*}
for $1\leq i\leq m-1$ and $A=H,X^\pm$. Thus, we can consider that $\Psi_2$ is the affine analogue of $\Psi_2^{\text{fin}}$.

\appendix
\section{Some formulas for the proof of Theorem~\ref{Main} and Theorem~\ref{Main-theorem}}
For the proof of Theorem~\ref{Main-theorem}, we note one lemma.
\begin{Lemma}
For $a,b\geq0$, we obtain
\begin{align}
&\quad[E_{j+n,u}t^{-s-a}E_{u,j+n}t^{s+a},E_{i,k}t^{-v-1} E_{k,i}t^{v+1}]\nonumber\\
&=E_{j+n,u}t^{-s-a}(\delta_{i,j+n}E_{u,k}t^{s-v+a-1}-\delta_{k,u}E_{i,j+n}t^{s-v+a-1})E_{k,i}t^{v+1}\nonumber\\
&\quad+(\delta_{u,i}E_{j+n,k}t^{-s-v-a-1}-\delta_{j+n,k}E_{i,u}t^{-s-v-a-1})E_{u,j+n}t^{s+a}E_{k,i}t^{v+1}\nonumber\\
&\quad+E_{i,k}t^{-v-1} E_{j+n,u}t^{-s-a}(\delta_{j+n,k}E_{u,i}t^{s+v+a+1}-\delta_{u,i}E_{k,j+n}t^{s+v+a+1})\nonumber\\
&\quad+E_{i,k}t^{-v-1}(\delta_{u,k}E_{j+n,i}t^{v-s-a+1}-\delta_{i,j+n}E_{k,u}t^{v-s-a+1})E_{u,j+n}t^{s+a},\label{AP-1}\\
&\quad[E_{u,j+n}t^{-s-a}E_{j+n,u}t^{s+a},E_{i,k}t^{-v-1} E_{k,i}t^{v+1}]\nonumber\\
&=E_{u,j+n}t^{-s-a}(\delta_{u,i}E_{j+n,k}t^{s+a-v-1}-\delta_{k,j+n}E_{i,u}t^{s+a-v-1})E_{k,i}t^{v+1}\nonumber\\
&\quad+(\delta_{i,j+n}E_{u,k}t^{-s-v-a-1}-\delta_{k,u}E_{i,j+n}t^{-s-v-a-1})E_{j+n,u}t^{s+a}E_{k,i}t^{v+1}\nonumber\\
&\quad+E_{i,k}t^{-v-1} E_{u,j+n}t^{-s-a}(\delta_{u,k}E_{j+n,i}t^{s+v+a+1}-\delta_{i,j+n}E_{k,u}t^{s+v+a+1})\nonumber\\
&\quad+E_{i,k}t^{-v-1} (\delta_{j+n,k}E_{u,i}t^{v-s-a+1}-\delta_{i,u}E_{k,j+n}t^{v-s-a+1})E_{j+n,u}t^{s+a},\label{AP-2}\\
&\quad[E_{u,i}t^{-s-a}E_{i,u}t^{s+a},E_{l,j+n}t^{-v-b}E_{j+n,l}t^{v+b}]\nonumber\\
&=E_{u,i}t^{-s-a}(\delta_{u,l}E_{i,j+n}t^{s-v+a-b}-\delta_{i,j+n}E_{l,u}t^{s-v+a-b})E_{j+n,l}t^{v+b}\nonumber\\
&\quad+(\delta_{l,i}E_{u,j+n}t^{-s-v-a-b}-\delta_{u,j+n}E_{l,i}t^{-s-v-a-b})E_{i,u}t^{s+a}E_{j+n,l}t^{v+b}\nonumber\\
&\quad+E_{l,j+n}t^{-v-b}E_{u,i}t^{-s-a}(\delta_{u,j+n}E_{i,l}t^{s+v+a+b}-\delta_{i,l}E_{j+n,u}t^{s+v+a+b})\nonumber\\
&\quad+E_{l,j+n}t^{-v-b}(\delta_{i,j+n}E_{u,l}t^{v-s+b-a}-\delta_{u,l}E_{j+n,i}t^{v-s+b-a})E_{i,u}t^{s+a},\label{AQ-1}\\
&\quad[E_{i,u}t^{-s-a}E_{u,i}t^{s+a},E_{l,j+n}t^{-v-b}E_{j+n,l}t^{v+b}]\nonumber\\
&=E_{i,u}t^{-s-a}(\delta_{i,l}E_{u,j+n}t^{s-v+a-b}-\delta_{j+n,u}E_{l,i}t^{s-v+a-b})E_{j+n,l}t^{v+b}\nonumber\\
&\quad+(\delta_{u,l}E_{i,j+n}t^{-s-v-a-b}-\delta_{i,j+n}E_{l,u}t^{-s-v-a-b})E_{u,i}t^{s+a}E_{j+n,l}t^{v+b}\nonumber\\
&\quad+E_{l,j+n}t^{-v-b}E_{i,u}t^{-s-a}(\delta_{i,j+n}E_{u,l}t^{s+v+a+b}-\delta_{u,l}E_{j+n,i}t^{s+v+a+b})\nonumber\\
&\quad+E_{l,j+n}t^{-v-b}(\delta_{u,j+n}E_{i,l}t^{v-s+b-a}-\delta_{i,l}E_{j+n,u}t^{v-s+b-a})E_{u,i}t^{s+a}.\label{AQ-2}
\end{align}
\end{Lemma}
The proof is due to a direct computation.

\section*{Acknowledgement}
The author expresses his sincere thanks to Thomas Creutzig and Nicolas Guay for the helpful discussion.

\section*{Funding}
This work was supported by JSPS Overseas Research Fellowships, Grant Number JP2360303. 
\section*{Data Availability}
The authors confirm that the data supporting the findings of this study are available within the article and its supplementary materials.
\section*{Conflicts of interests/Competing interests}
The authors have no competing interests to declare that are relevant to the content of this article.
\bibliographystyle{plain}
\bibliography{syuu}
\end{document}